\documentclass{article}

\usepackage[final,nonatbib]{neurips_2019}

\usepackage{amsthm,amsmath,mathtools,bm,enumerate,subcaption,amssymb,dsfont}

\usepackage[utf8]{inputenc} 
\usepackage[T1]{fontenc}    
\usepackage{hyperref}       
\usepackage{url}            
\usepackage{booktabs}       
\usepackage{amsfonts}       
\usepackage{nicefrac}       
\usepackage{microtype}      

\usepackage{color}     
\usepackage{graphicx}
\usepackage{standalone}
\usepackage{algorithm}
\usepackage{tikz}
\usetikzlibrary{patterns}
\usetikzlibrary{shadows.blur}
\usetikzlibrary{matrix}
\usetikzlibrary{calc}
\newcommand{\EE}{\mathbb{E}}

\usepackage[mathcal]{eucal}
\DeclareMathOperator{\rank}{rank}

\newcommand{\RR}{\mathbb{R}}
\newcommand{\NN}{\mathbb{N}}

\newcommand{\Ff}{\mathcal{F}}

\newcommand{\Nn}{\mathcal{N}}

\newcommand{\Pp}{\mathcal{P}}

\newcommand{\Ww}{\mathcal{W}}
\newcommand{\Xx}{\mathcal{X}}

\newcommand{\Id}{\mathrm{Id}}

\newcommand{\Lip}{\mathrm{Lip}}

\DeclareMathOperator{\sign}{sign}

\newtheorem{theorem}{Theorem}[section]
\newtheorem{proposition}[theorem]{Proposition}
\newtheorem{lemma}[theorem]{Lemma}

\newtheorem{assumption}[theorem]{Assumption}

\graphicspath{ {images/} }

\title{On Lazy Training in  Differentiable Programming}

\author{
  L\'ena\"ic Chizat\\
  CNRS, Universit\'e Paris-Sud\\
  Orsay, France\\
  \texttt{lenaic.chizat@u-psud.fr} \\
   \And
   Edouard Oyallon\\
   CentraleSupelec, INRIA\\
   Gif-sur-Yvette, France \\
   \texttt{edouard.oyallon@centralesupelec.fr} \\
  \And
   Francis Bach \\
   INRIA, ENS, PSL Research University \\
   Paris, France \\
   \texttt{francis.bach@inria.fr} \\
}

\begin{document}

\maketitle

\begin{abstract}
In a series of recent theoretical works, it was shown that strongly over-parameterized neural networks trained with gradient-based methods could converge exponentially fast to zero training loss, with their parameters hardly varying. In this work, we show that this ``lazy training'' phenomenon is not specific to over-parameterized neural networks, and is due to a choice of scaling, often implicit, that makes the model behave as its linearization around the initialization, thus yielding a model equivalent to learning with positive-definite kernels. Through a theoretical analysis, we exhibit various situations where this phenomenon arises in non-convex optimization and we provide bounds on the distance between the lazy and linearized optimization paths. Our numerical experiments bring a critical note, as we observe that the performance of commonly used non-linear deep convolutional neural networks in computer vision degrades when trained in the lazy regime. This makes it unlikely that ``lazy training'' is behind the many successes of neural networks in difficult high dimensional tasks.
\end{abstract}


\section{Introduction}\label{sec:intro}
Differentiable programming is becoming an important paradigm in signal processing and machine learning that consists in building parameterized models, sometimes with a complex architecture and a large number of parameters, and adjusting these parameters in order to minimize a loss function using gradient-based optimization methods. The resulting problem is in general highly non-convex. It has been observed empirically that, for fixed loss and model class, changes in the parameterization, optimization procedure, or initialization could lead to a selection of models with very different properties~\cite{zhang2016understanding}. This paper is about one such implicit bias phenomenon, that we call \emph{lazy training}, which corresponds to the model behaving like its linearization around the initialization.

This work is motivated by a series of recent articles~\cite{du2018gradient, li2018learning, du2018gradientdeep, allenzhu2018convergence, zou2018sgdoptimizes} where it is shown that over-parameterized neural networks could converge linearly to zero training loss with their parameters hardly varying. With a slightly different approach, it was shown in~\cite{jacot2018neural} that infinitely wide neural networks behave like the linearization of the neural network around its initialization. In the present work, we argue that this behavior is not specific to neural networks, and is not so much due to over-parameterization than to an implicit choice of scaling. By introducing an explicit scale factor, we show that essentially any parametric model can be trained in this lazy regime if its output is close to zero at initialization. This shows that guaranteed fast training is indeed often possible, but at the cost of recovering a linear method\footnote{Here we mean a prediction function linearly parameterized by a potentially infinite-dimensional vector.}. Our experiments on two-layer neural networks and deep convolutional neural networks (CNNs) suggest that this behavior is undesirable in practice. 

\subsection{Presentation of lazy training}
We consider a parameter space\footnote{Our arguments could be generalized to the case where the parameter space is a Riemannian manifold.} $\RR^p$, a Hilbert space $\Ff$, a smooth model $h : \RR^p \to \Ff$ (such as a neural network) and a smooth loss  $R:\Ff\to \RR_+$. We aim to minimize, with gradient-based methods, the objective function $F:\RR^p\to \RR_+$ defined as
\[
F(w) \coloneqq R(h(w)).
\]
With an initialization $w_0\in \RR^p$, we define the linearized model $\bar h(w) = h(w_0) + Dh(w_0)(w-w_0)$ around $w_0$, and the corresponding objective $\bar F : \RR^p\to \RR_+$ as
\[
\bar F(w) \coloneqq R(\bar h(w)).
\]
It is a general fact that the optimization path of $F$ and $\bar F$ starting from $w_0$ are close at the beginning of training. We call \emph{lazy training} the less expected situation where these two paths remain close until the algorithm is stopped.

Showing that a certain non-convex optimization is in the lazy regime opens the way for surprisingly precise results, because linear models are rather well understood. For instance, when $R$ is strongly convex, gradient descent on $\bar F$ with an appropriate step-size converges linearly to a global minimizer~\cite{bottou2018optimization}. For two-layer neural networks, we show in Appendix~\ref{app:randomfeature} that the linearized model is a random feature model~\cite{rahimi2008random} which lends itself nicely to statistical analysis~\cite{carratino2018learning}. Yet, while advantageous from a theoretical perspective, it is not clear \emph{a priori} whether this lazy regime is desirable in practice.

This phenomenon is illustrated in Figure~\ref{fig:cover} where lazy training for a two-layer neural network with rectified linear units (ReLU) is achieved by increasing the variance $\tau^2$ at initialization (see next section). While in panel~(a) the ground truth features are identified, this is not the case for lazy training on panel~(b) that manages to interpolate the observations with just a small displacement in parameter space (in both cases, near zero training loss was achieved). As seen on panel~(c), this behavior hinders good generalization in the teacher-student setting~\cite{saad1995line}. The plateau reached for large $\tau$ corresponds exactly to the performance of the linearized model, see Section~\ref{sec:synthexpe} for details.
\begin{figure}[t]
\centering
\begin{subfigure}{0.3\linewidth}
\centering
\includegraphics[scale=0.4]{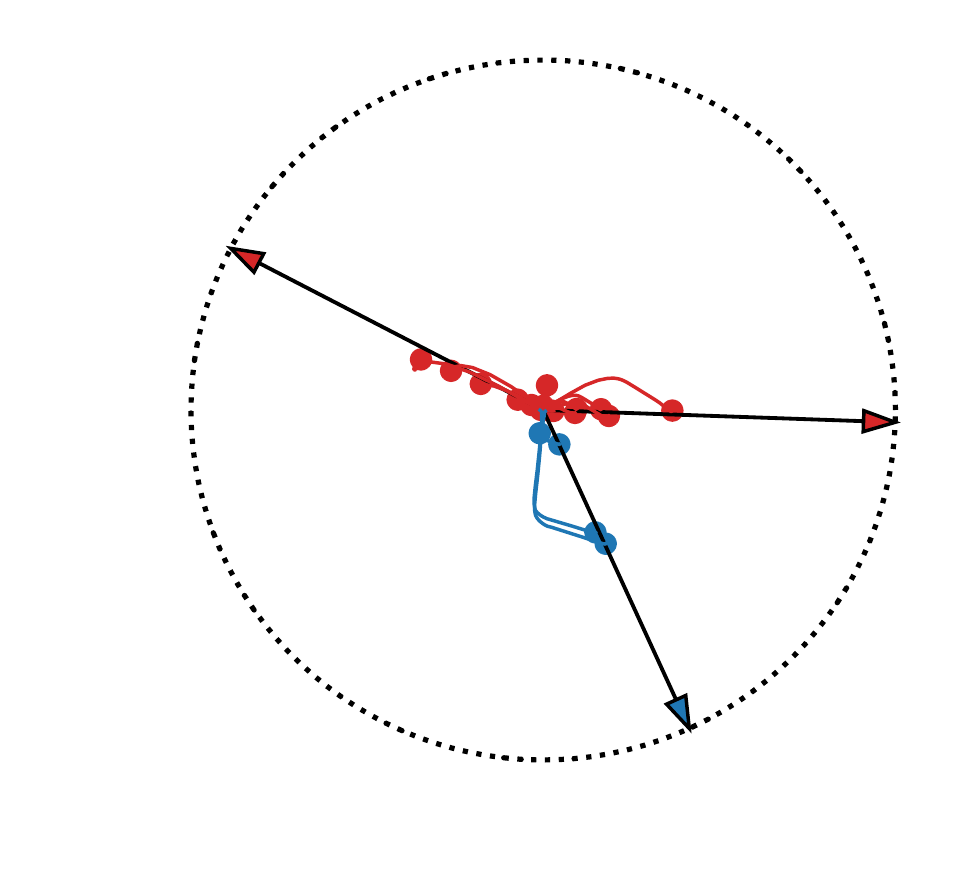}
\caption{Non-lazy training ($\tau=0.1$)}
\end{subfigure}%
\begin{subfigure}{0.3\linewidth}
\centering
\includegraphics[scale=0.4]{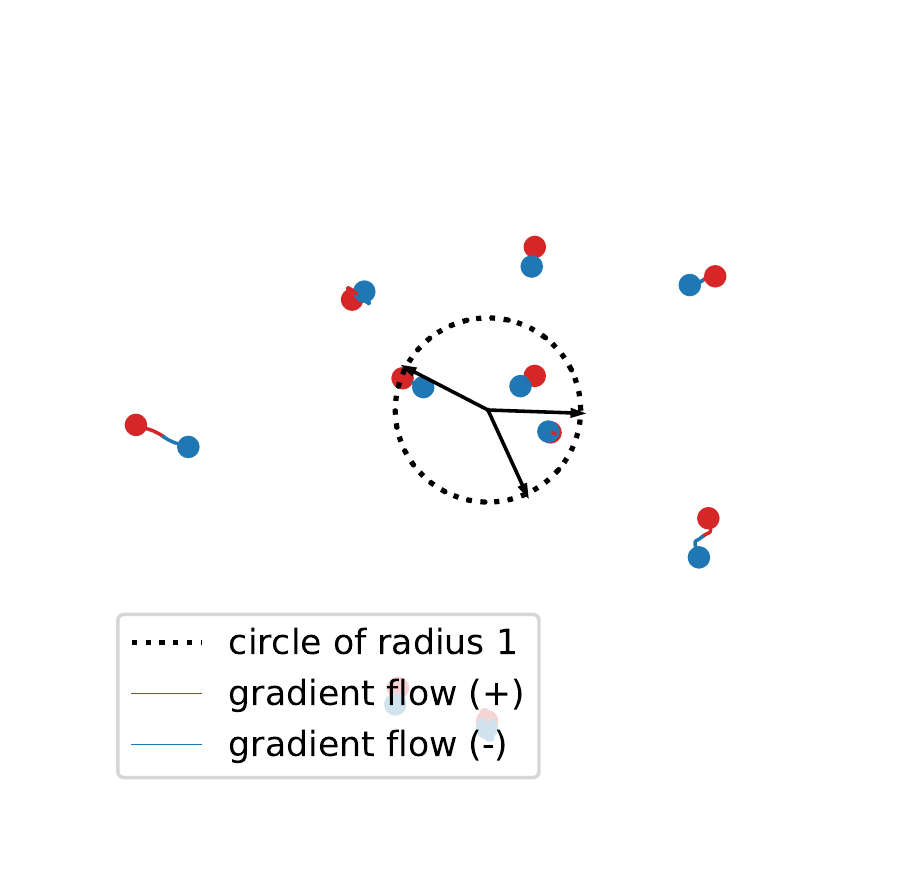}
\caption{Lazy training ($\tau=2$)}
\end{subfigure}%
\begin{subfigure}{0.3\linewidth}
\centering
\includegraphics[scale=0.37]{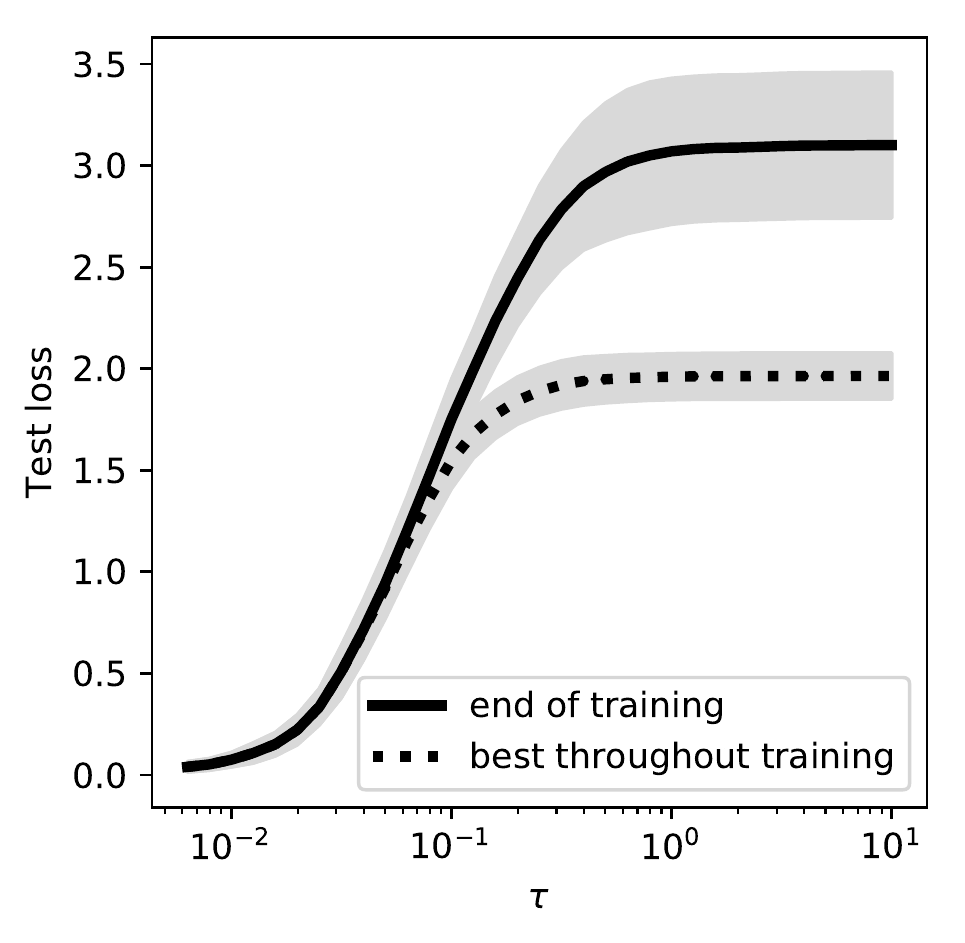}
\caption{Generalization properties}
\end{subfigure}
\caption{Training a two-layer ReLU neural network initialized with normal random weights of variance $\tau^2$: lazy training occurs when $\tau$ is large. (a)-(b) Trajectory of weights during gradient descent in $2$-D (color shows sign of output layer). (c) Generalization in $100$-D: it worsens as $\tau$ increases. The ground truth is generated with $3$ neurons (arrows in (a)-(b)). Details in Section~\ref{sec:numerics}.}
\label{fig:cover}
\end{figure}

\subsection{When does lazy training occur?}\label{sec:range}

\paragraph{A general criterion.} Let us start with a formal computation. We assume that $w_0$ is not a minimizer so that $F(w_0)>0$, and not a critical point so that $\nabla F(w_0)\neq 0$. Consider a gradient descent step $w_1 \coloneqq w_0 - \eta \nabla F(w_0)$, with a small stepsize $\eta>0$. On the one hand, the relative change of the objective is
$
\Delta(F) \coloneqq \frac{\vert F(w_1)-F(w_0)\vert }{F(w_0)} \approx \eta \frac{\Vert \nabla F(w_0)\Vert^2}{F(w_0)}.
$
On the other hand, the relative change of the differential of $h$ measured in operator norm is
$
\Delta (Dh) \coloneqq \frac{\Vert Dh(w_1) - Dh(w_0)\Vert}{\Vert Dh(w_0)\Vert} \leq \eta \frac{\Vert \nabla F(w_0)\Vert\cdot \Vert D^2h(w_0)\Vert}{\Vert Dh(w_0)\Vert}
$. Lazy training refers to the case where the differential of $h$ does not sensibly change while the loss enjoys a significant decrease, i.e.,  $\Delta(F) \gg \Delta(Dh)$. Using the above estimates, this is guaranteed when
\begin{align*}
 \frac{\Vert \nabla F(w_0)\Vert}{F(w_0)}  \gg \frac{\Vert D^2h(w_0)\Vert}{\Vert Dh(w_0)\Vert}.
\end{align*}
For the square loss $R(y)=\frac12 \Vert y - y^\star\Vert^2$ for some $y^\star\in \Ff$, this leads to the simpler criterion
\begin{align}\label{eq:lazycriterionsimple}
\kappa_h(w_0) \coloneqq \Vert h(w_0)-y^\star\Vert\frac{ \Vert D^2h(w_0)\Vert} {\Vert Dh(w_0)\Vert^2}\ll 1,
\end{align}
using the approximation $\Vert \nabla F(w_0) \Vert = \Vert Dh(w_0)^\intercal (h(w_0)-y^\star) \Vert \approx \Vert Dh(w_0)\Vert \cdot \Vert h(w_0)-y^\star\Vert$. This quantity $\kappa_h(w_0)$ could be called the inverse \emph{relative scale} of the model $h$ at $w_0$. We prove in Theorem~\ref{th:square} that it indeed controls how much the training dynamics differs from the linearized training dynamics when $R$ is the square loss\footnote{Note that lazy training could occur even when $\kappa_h(w_0)$ is large, i.e. Eq.~\eqref{eq:lazycriterionsimple} only gives a sufficient condition.}. For now, let us explore situations in which lazy training can be shown to occur, by investigating the behavior of $\kappa_h(w_0)$.

\paragraph{Rescaled models.} Considering a scaling factor $\alpha>0$, it holds
\[
\kappa_{\alpha h}(w_0) = \frac1\alpha \Vert \alpha h(w_0) -y^\star\Vert \frac{ \Vert D^2h(w_0)\Vert}{\Vert Dh(w_0)\Vert^2}.
\]
Thus, $\kappa_{\alpha h}(w_0)$ simply decreases as $\alpha^{-1}$ when $\alpha$ grows and $\Vert \alpha h(w_0)-y^\star\Vert$ is bounded, leading to lazy training for large~$\alpha$. Training dynamics for such rescaled models are studied in depth in Section~\ref{sec:dynamics}. For neural networks, there are various ways to ensure $h(w_0)=0$, see Section~\ref{sec:numerics}.

\paragraph{Homogeneous models.}
If $h$ is $q$-positively homogeneous\footnote{That is, for  $q\geq 1$, it holds $h(\lambda w) = \lambda^q h(w)$ for all $\lambda>0$ and $w\in \RR^p$.} then multiplying the initialization by $\lambda$  is equivalent to multiplying the scale factor $\alpha$ by $\lambda^q$. In equation,
\[
\kappa_{h}(\lambda w_0) = \frac1{\lambda^q}\Vert \lambda^q h(w_0) -y^\star\Vert\frac{ \Vert D^2h(w_0)\Vert}{\Vert Dh(w_0)\Vert^2}.
\]
This formula applies for instance to $q$-layer neural networks consisting of a cascade of homogenous non-linearities and linear, but not affine, operators. Such networks thus enter the lazy regime as the variance of initialization increases, if one makes sure that the initial output has bounded norm (see Figures~\ref{fig:cover} and~\ref{fig:losses}(b) for $2$-homogeneous examples).

\paragraph{Two-layer neural networks.} \label{par:wide}
For $m, d \in \NN$, consider functions $h_m : (\RR^d)^m\to \Ff$ of the form
\[
h_m(w) = \alpha(m) \sum_{i=1}^m \phi(\theta_i),
\]
where $\alpha(m)>0$ is a normalization, $w=(\theta_1,\dots,\theta_m)$ and $\phi:\RR^d\to \Ff$ is a smooth function. This setting covers the case of two-layer neural networks (see Appendix~\ref{app:randomfeature}). When initializing with independent and identically distributed variables $(\theta_i)_{i=1}^m$ satisfying $\EE \phi(\theta_i)=0$, and under the assumption that $D\phi$ is not identically $0$ on the support of the initialization, we prove in Appendix~\ref{app:randomfeature} that for large $m$ it holds
\[
\EE[\kappa_{h_m}(w_0)] \lesssim m^{-\frac12} +(m\alpha(m))^{-1}.
\]
As a consequence, as long as $m \alpha(m)\to \infty$ when $m\to \infty$, such models are bound to reach the lazy regime. In this case, the norm of the initial output becomes negligible in front of the scale as $m$ grows due to the statistical cancellations that follow from the assumption $\EE\phi(\theta_i)=0$. In contrast, the critical scaling $\alpha(m)=1/m$, allows to converge as $m\to \infty$ to a non degenerate dynamic described by a partial differential equation and referred to as the mean-field limit~\cite{mei2018mean,chizat2018global,rotskoff2018neural,sirignano2018mean}. 

\subsection{Content and contributions}
\label{sec:content}

The goal of this paper is twofold: (i) understanding in a general optimization setting when lazy training occurs, and (ii) investigating the practical usefulness of models in the lazy regime. It  is organized as follows:
\begin{itemize}
\item in Section~\ref{sec:dynamics}, we study the gradient flows for rescaled models $\alpha h$ and prove in various situations that for large $\alpha$, they are close to gradient flows of the linearized model. When the loss is strongly convex, we also prove that lazy gradient flows converge linearly, either to a global minimizer for over-parameterized models, or to a local minimizer for under-parameterized models.
\item in Section~\ref{sec:numerics}, we use numerical experiments on synthetic cases to illustrate how lazy training differs from other regimes of training (see also Figure~\ref{fig:cover}). Most importantly, we show empirically that CNNs used in practice could be far from the lazy regime, with their performance not exceeding that of some classical linear methods as they become lazy.
\end{itemize}
Our focus is on general principles and qualitative description.

\paragraph{Updates of the paper.} This article is an expanded version of ``A Note on Lazy Training in Supervised Differential Programming'' that appeared online in December 2018. Compared to the first version, it has been complemented with finite horizon bounds in Section~\ref{sec:finitehorizon} and numerical experiments on CNNs in Section~\ref{sec:CNNexpe} while the rest of the material has just been slightly reorganized.
%


\section{Analysis of Lazy Training Dynamics}\label{sec:dynamics}

\subsection{Theoretical setting}
Our goal in this section is to show that lazy training dynamics for the scaled objective
 \begin{equation}\label{eq:scaledobjective}
F_\alpha(w) \coloneqq \frac{1}{\alpha^2}R(\alpha h(w))
\end{equation}
 are close, when the scaling factor $\alpha$ is large, to those of the scaled objective for the linearized model
 \begin{equation}\label{eq:scaledobjectivelinear}
\bar F_\alpha(w) \coloneqq \frac{1}{\alpha^2}R(\alpha \bar h(w)) ,
\end{equation}
where $\bar h(w)\coloneqq h(w_0) + Dh(w_0)(w-w_0)$ and $w_0 \in \RR^p$ is a fixed initialization. Multiplying the objective by $1/\alpha^2$ does not change the minimizers, and corresponds to the proper time parameterization of the dynamics for large $\alpha$. Our basic assumptions are the following:
\begin{assumption} The parametric model $h:\RR^p\to \Ff$ is differentiable  with a locally Lipschitz differential\footnote{$Dh(w)$ is a continuous linear map from $\RR^p$ to $\Ff$. The Lipschitz constant of $Dh:w\mapsto Dh(w)$ is defined with respect to the operator norm. When $\Ff$ has a finite dimension, $Dh(w)$ can be identified with the Jacobian matrix of $h$ at $w$.} $Dh$. Moreover, $R$ is differentiable with a Lipschitz gradient.
\end{assumption}
This setting is mostly motivated by supervised learning problems, where one considers a probability distribution $\rho \in \Pp(\RR^d\times \RR^k)$ and defines $\Ff$ as the space $L^2(\rho_x;\RR^k)$ of square-integrable functions with respect to $\rho_x$, the marginal of $\rho$ on $\RR^d$. The risk $R$ is then built from a smooth loss function $\ell: (\RR^k)^2 \to \RR_+$ as $R(g)=\EE_{(X,Y)\sim \rho} \ell(g(X),Y)$. This corresponds to empirical risk minimization when $\rho$ is a finite discrete measure, and to population risk minimization otherwise (in which case only stochastic gradients are available to algorithms). Finally, one defines $h(w) = f(w,\cdot)$ where $f:\RR^p\times \RR^d\to \RR^k$ is a parametric model, such as a neural network, which outputs in $\RR^k$ depend on parameters in $\RR^p$ and input data in $\RR^d$. 

\paragraph{Gradient flows.} 
In the rest of this section, we study the \emph{gradient flow} of the objective function $F_\alpha$ which is an approximation of (accelerated) gradient descent~\cite{gautschi1997numerical, scieur2017integration} and stochastic gradient descent~\cite[Thm. 2.1]{kushner2003stochastic} with small enough step sizes.
With an initialization $w_0\in \RR^p$, the gradient flow of $F_\alpha$ is the path $(w_\alpha(t))_{t\geq 0}$ in the space of parameters $\RR^p$ that satisfies $w_\alpha(0)=w_0$ and solves the ordinary differential equation
\begin{equation}\label{eq:lazyGF}
w'_\alpha(t) = - \nabla F_\alpha(w_\alpha(t)) =-  \frac1\alpha Dh(w_\alpha(t))^\intercal \nabla R(\alpha h(w_\alpha(t))) ,
\end{equation}
where $Dh^\intercal$ denotes the adjoint of the differential $Dh$. We will study this dynamic for itself, and will also compare it to the gradient flow $(\bar w_\alpha(t))_{t\geq 0}$ of $\bar F_\alpha$ that satisfies $\bar w_\alpha (0)=w_0$ and solves
\begin{equation}\label{eq:linearizedGF}
\bar w'_\alpha(t) = - \nabla \bar F_\alpha(\bar w_\alpha(t)) = -  \frac1\alpha Dh(w_0)^\intercal \nabla R(\alpha \bar h(\bar w_\alpha(t))).
\end{equation}
Note that when $h(w_0)=0$, the renormalized dynamic $w_0 + \alpha (\bar w_\alpha(t)-w_0)$ does not depend on $\alpha$, as it simply follows the gradient flow of $w\mapsto R(Dh(w_0)(w-w_0))$ starting from $w_0$.

\subsection{Bounds with a finite time horizon}\label{sec:finitehorizon}
We start with a general result  that confirms that when $h(w_0)=0$, taking large $\alpha$ leads to lazy training. We do not assume convexity of $R$.
\begin{theorem}[General lazy training]\label{th:generallazy}
Assume that $h(w_0)=0$. Given a fixed time horizon $T>0$, it holds $\sup_{t\in [0,T]} \Vert w_\alpha(t)-w_0\Vert =O(1/\alpha)$,
\begin{align*}
\sup_{t\in [0,T]} \Vert w_\alpha(t) -\bar w_\alpha(t)\Vert = O(1/\alpha^2) &&\text{and}&& \sup_{t\in [0,T]} \Vert \alpha h(w_\alpha(t)) - \alpha \bar h (\bar w_\alpha(t))\Vert = O(1/\alpha).
\end{align*}
\end{theorem}
For supervised machine learning problems, the bound on $\Vert w_\alpha(t)-\bar w_\alpha(t)\Vert$ implies that $\alpha h(w_\alpha(T))$ also \emph{generalizes} like $\alpha \bar h(\bar w_\alpha(T))$ outside of the training set for large $\alpha$, see Appendix~\ref{subsec:generalization}. Note that the generalization behavior of linear models has been widely studied, and is particularly well understood for random feature models~\cite{rahimi2008random}, which are recovered when linearizing two layer neural networks, see Appendix~\ref{app:randomfeature}. It is possible to track the constants in Theorem~\ref{th:generallazy} but they would depend exponentially on the time horizon $T$. This exponential dependence can however be discarded for the specific case of the square loss, where we recover the \emph{scale} criterion informally derived in Section~\ref{sec:range}.
\begin{theorem}[Square loss, quantitative] \label{th:square} Consider the square loss $R(y) =\frac12\Vert y-y^\star\Vert^2$ for some $y^\star\in \Ff$ and assume that for some (potentially small) $r>0$,  $h$ is $\Lip(h)$-Lipschitz and $Dh$ is $\Lip(Dh)$-Lipschitz on the ball of radius $r$ around $w_0$.  Then for an iteration number $K>0$ and corresponding time $T \coloneqq K/\Lip(h)^2$, it holds
\[
\frac{\Vert \alpha h(w_\alpha(T)) - \alpha \bar h(\bar w_\alpha(T))\Vert}{\Vert \alpha h(w_0) -y^\star\Vert} \leq \frac{K^2}{\alpha} \frac{\Lip(Dh)}{ \Lip(h)^2} \Vert \alpha h(w_0) -y^\star\Vert
\]
as long as $\alpha\geq K\Vert \alpha h(w_0)-y^\star\Vert/(r\Lip(h))$.
\end{theorem}

We can make the following observations:
\begin{itemize}
\item For the sake of interpretability, we have introduced a quantity $K$, analogous to an iteration number, that accounts for the fact that the gradient flow needs to be integrated with a step-size of order  $1/\Lip(\nabla F_\alpha)= 1/\Lip(h)^2$. For instance, with this step-size, gradient descent at iteration $K$ approximates the gradient flow at time $T=K/\Lip(h)^2$, see, e.g.,~\cite{gautschi1997numerical,scieur2017integration}.
\item Laziness only depends on the local properties of $h$ around $w_0$. These properties may vary a lot over the parameter space, as is the case for homogeneous functions seen in Section~\ref{sec:range}.
\end{itemize}

For completeness, similar bounds on $\Vert w_\alpha(T) -w_0\Vert$ and $\Vert w_\alpha(T)-\bar w_\alpha(T)\Vert$ are also provided in Appendix~\ref{app:squareproof}. The drawback of the bounds in this section is the increasing dependency in time, which is removed in the next section. Yet, the relevance of Theorem~\ref{th:generallazy} remains because it does not depend on the conditioning of the problem. Although the bound grows as $K^2$, it gives an informative estimate for large or ill-conditioned problems, where training is typically stopped much before convergence.

\subsection{Uniform bounds and convergence in the lazy regime}
This section is devoted to uniform bounds in time and convergence results under the assumption that $R$ is strongly convex. In this setting, the function $\bar F_\alpha$ is strictly convex on the affine hyperspace $w_0 + \ker Dh(w_0)^\perp$ which contains the linearized gradient flow $(\bar w_\alpha(t))_{t\geq 0}$, so the latter converges linearly to the unique global minimizer of $\bar F_\alpha$. In particular, if $h(w_0)=0$ then this global minimizer does not depend on $\alpha$ and $\sup_{t\geq 0} \Vert \bar w_\alpha(t)-w_0\Vert = O(1/\alpha)$. We will see in this part how these properties reflect on the gradient flow $w_\alpha(t)$ for large $\alpha$.

\paragraph{Over-parameterized case.}
The following proposition shows global convergence of lazy training under the condition that $Dh(w_0)$ is surjective. As $\rank Dh(w_0)$ gives the number of effective parameters or degrees of freedom of the model around $w_0$, this over-parameterization assumption guarantees that any model around $h(w_0)$ can be fitted. Of course, this can only happen if $\Ff$ is finite-dimensional.

\begin{theorem}[Over-parameterized lazy training]\label{th:over}
Consider a $M$-smooth and $m$-strongly convex loss $R$ with minimizer $y^\star$ and condition number $\kappa \coloneqq M/m$. Assume that $\sigma_{\min}$, the smallest singular value of $Dh(w_0)^\intercal$ is positive and that the initialization satisfies $\Vert h(w_0)\Vert \leq C_0 \coloneqq \sigma_{\min}^3/( 32 \kappa^{3/2} \Vert D h(w_0)\Vert \Lip(Dh))$ where $\Lip(Dh)$ is the Lipschitz constant of $Dh$. If $\alpha > \Vert y^*\Vert/C_0$, then for $t\geq 0$, it holds
\[
\Vert \alpha h(w_\alpha(t))-y^* \Vert \leq \sqrt{\kappa}\Vert \alpha h(w_0) - y^*\Vert \exp(- m \sigma_{\min}^2 t/4).
\]
If moreover $h(w_0)=0$,  it holds as $\alpha\to \infty$, $\sup_{t\geq 0} \Vert w_\alpha(t) - w_0 \Vert = O(1/\alpha)$,
\begin{align*}
\sup_{t\geq 0} \Vert \alpha h(w_\alpha(t)) - \alpha \bar h(\bar w_\alpha(t)) \Vert = O(1/\alpha)
&&
\text{and}
&&
\sup_{t\geq 0} \Vert w_\alpha(t) - \bar w_\alpha(t) \Vert = O(\log \alpha /\alpha^2).
\end{align*}
\end{theorem}

The proof of this result relies on the fact that $\alpha h(w_\alpha(t))$ follows the gradient flow of $R$ in a time-dependent and non degenerate metric: the pushforward metric~\cite{lee2003smooth} induced by~$h$ on $\Ff$. For the first part, we do not claim improvements over~\cite{du2018gradient, li2018learning, du2018gradientdeep, allenzhu2018convergence, zou2018sgdoptimizes}, where a lot of effort is also put in dealing with the non-smoothness of $h$, which we do not study here. As for the uniform in time comparison with the tangent gradient flow, it is new and follows mostly from Lemma~\ref{lem:stability} in Appendix~\ref{app:proofs} where the constants are given and depend polynomially on the characteristics of the problem.

\paragraph{Under-parameterized case.}
We now remove the over-parameterization assumption and show again linear convergence for large values of $\alpha$. This covers in particular the case of population loss minimization, where $\Ff$ is infinite-dimensional. For this setting, we limit ourselves to a qualitative statement\footnote{In contrast to the finite horizon bound of Theorem~\ref{th:square}, quantitative statements would here involve the smallest positive singular value of $Dh(w_0)$, which is anyways hard to control.}.

\begin{theorem}[Under-parameterized lazy training]\label{th:under} Assume that $\Ff$ is separable, $R$ is strongly convex, $h (w_0)=0$ and $\rank Dh(w)$ is constant on a neighborhood of $w_0$. Then there exists $\alpha_0>0$ such that for all $\alpha\geq \alpha_0$ the gradient flow~\eqref{eq:lazyGF} converges at a geometric rate (asymptotically independent of $\alpha$) to a local minimum of~$F_\alpha$.
\end{theorem}
Thanks to lower-semicontinuity of the rank function, the assumption that the rank is locally constant holds generically, in the sense that it is satisfied on an open dense subset of $\RR^p$. In this under-parameterized case, the limit $\lim_{t\to\infty} w_\alpha(t)$ is for $\alpha$ large enough a strict local minimizer, but in general not a global minimizer of $F_\alpha$ because the image of $Dh(w_0)$ does not \emph{a priori} contain the global minimizer of $R$. Thus it cannot be excluded that there exists parameters~$w$ farther from $w_0$ with a smaller loss. This fact is clearly observed experimentally in Section~\ref{sec:numerics}, Figure~\ref{fig:losses}-(b). Finally, a comparison with the linearized gradient flow as in Theorem~\ref{th:over} could be shown along the same lines, but would be technically slightly more involved because differential geometry comes into play.
 
\paragraph{Relationship to the global convergence result in~\cite{chizat2018global}.}
A consequence of Theorem~\ref{th:under} is that in the lazy regime, the gradient flow of the population risk for a two-layer neural network might get stuck in a local minimum. In contrast, it is shown in~\cite{chizat2018global} that such gradient flows converge to global optimality in the infinite over-parameterization limit $p\to \infty$ if initialized with enough diversity in the weights. This is not a contradiction since Theorem~\ref{th:under} assumes a finite number $p$ of parameters. In the lazy regime, the population loss might also converge to its minimum when $p$ increases: this is guaranteed if the tangent kernel $Dh(w_0)Dh(w_0)^\intercal $~\cite{jacot2018neural} converges (after normalization) to a universal kernel as $p\to\infty$. However, this convergence might be unreasonably slow in high-dimension, as  Figure~\ref{fig:cover}-(c) suggests. As a side note, we stress that the global convergence result in~\cite{chizat2018global} is not limited to lazy dynamics but also covers non-linear dynamics, such as seen on Figure~\ref{fig:cover} where neurons move.


\section{Numerical Experiments}
\label{sec:numerics}

We realized two sets of experiments, the first with two-layer neural networks conducted on synthetic data and the second with convolutional neural networks (CNNs) conducted on the CIFAR-10 dataset~\cite{krizhevsky2009learning}. The code to reproduce these experiments is available online\footnote{\url{https://github.com/edouardoyallon/lazy-training-CNN}}.
 
\subsection{Two-layer neural networks in the teacher-student setting}\label{sec:synthexpe}
We consider the following two-layer \emph{student} neural network $h_m(w) = f_m(w,\cdot)$ with $f_m(w,x) =  \sum_{j=1}^m a_j \max(b_j \cdot x,0)$ where $a_j \in \RR$ and $b_j \in \RR^d$ for $j=1,\dots,m$. It is trained to minimize the square loss with respect to the output of a two-layer \emph{teacher} neural network with same architecture and $m_0=3$ hidden neurons, with random weights normalized so that $\Vert a_j b_j \Vert = 1$ for $j\in \{1,2,3\}$. For the student network, we use random Gaussian weights, except when \emph{symmetrized initialization} is mentioned, in which case we use random Gaussian weights for $j\leq m/2$ and set for $j>m/2$, $b_j =b_{j-m/2}$ and $a_j=-a_{j-m/2}$. This amounts to training a model of the form $h(w_a,w_b)=h_{m/2}(w_a)-h_{m/2}(w_b)$ with $w_a(0)=w_b(0)$ and guaranties zero output at initialization. The training data are $n$ input points uniformly sampled on the unit sphere in $\RR^d$ and we minimize the empirical risk, except for Figure~\ref{fig:SGD}(b) where we directly minimize the population risk with Stochastic Gradient Descent (SGD).

\begin{figure}
\centering
\begin{subfigure}{0.5\linewidth}
\centering
\includegraphics[scale=0.43]{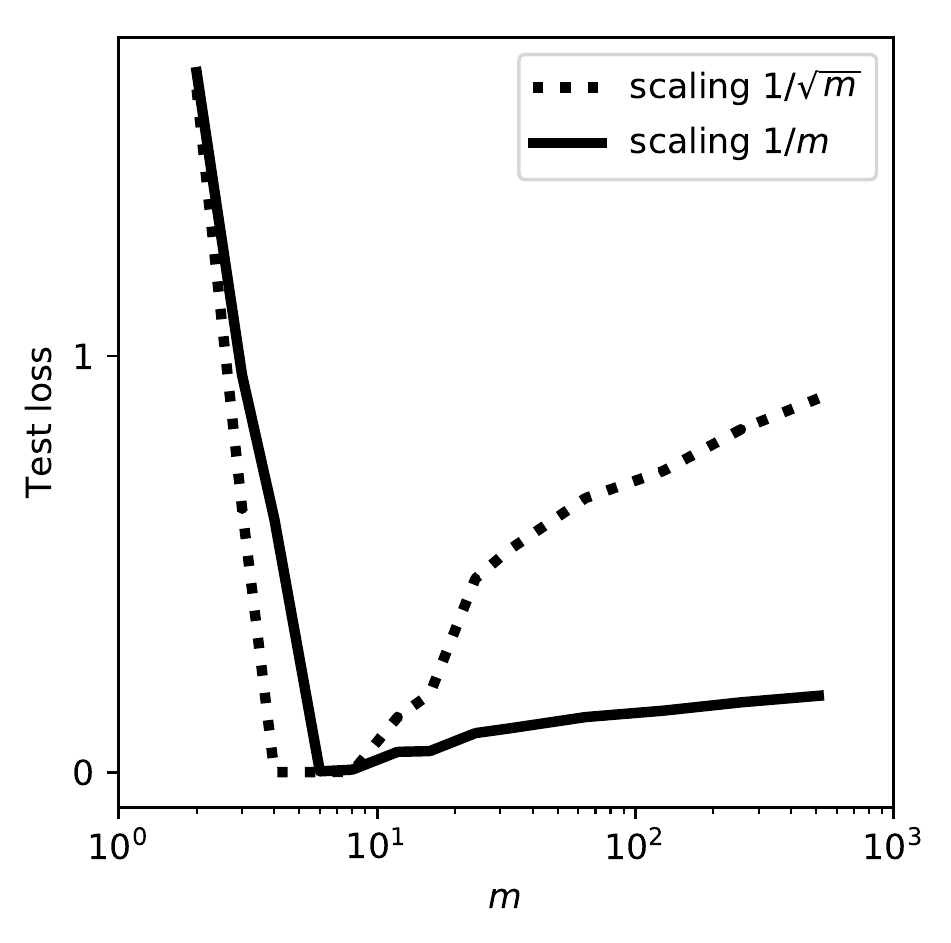}
\caption{}\label{fig:testm}
\end{subfigure}%
\begin{subfigure}{0.5\linewidth}
\centering
\includegraphics[scale=0.43]{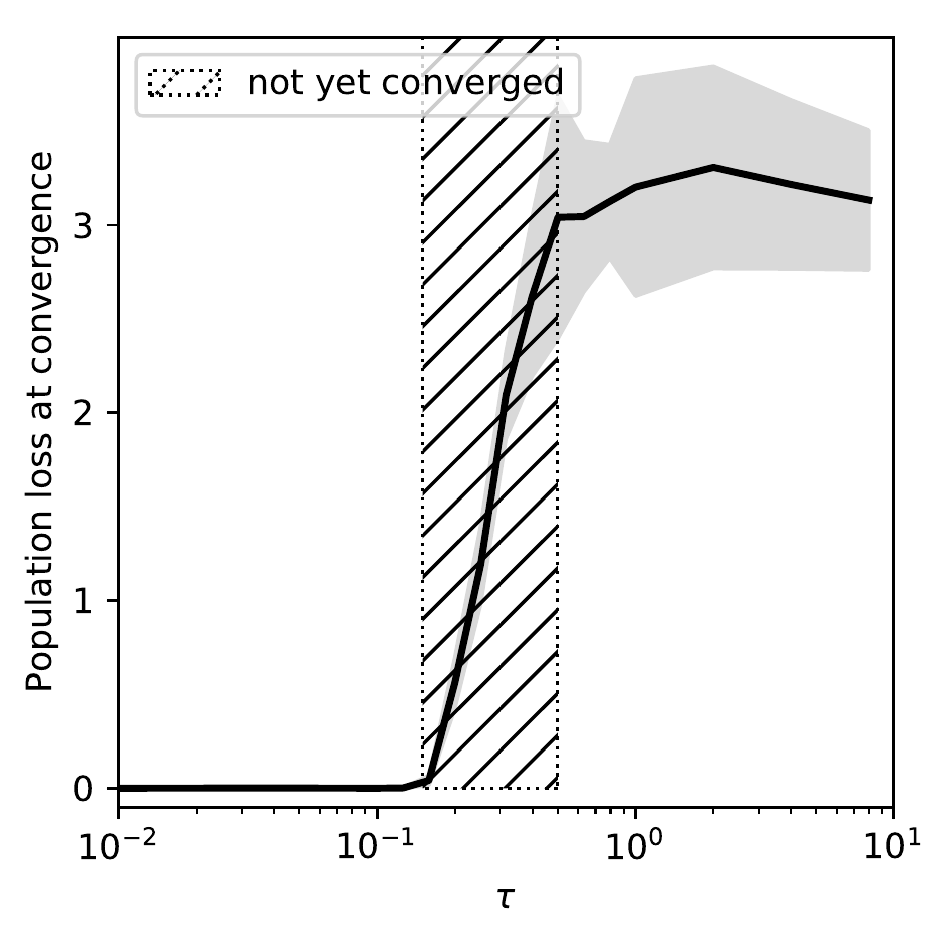}
\caption{}\label{fig:SGD}
\end{subfigure}
\caption{(a) Test loss at convergence for gradient descent, when $\alpha$ depends on $m$ as $\alpha = 1/m$ or $\alpha=1/\sqrt{m}$, the latter leading to lazy training for large $m$ (not symmetrized). (b) Population loss at convergence versus $\tau$ for SGD with a random $\Nn(0,\tau^2)$ initialization (symmetrized). In the hatched area the loss was still slowly decreasing.}
\label{fig:losses}
\end{figure}

\paragraph{Cover illustration.} Let us detail the setting of Figure~\ref{fig:cover} in Section~\ref{sec:intro}. Panels (a)-(b) show gradient descent dynamics with $n=15, m=20$ with symmetrized initialization (illustrations with more neurons can be found in Appendix~\ref{app:experiments}). To obtain a $2$-D representation, we plot $\vert a_j(t) \vert b_j(t)$ throughout training (lines) and at convergence (dots) for $j\in\{1,\dots,m\}$. The blue or red colors stand for the signs of $a_j(t)$ and the unit circle is displayed to help visualizing the change of scale.  On panel (c), we set $n=1000$, $m=50$ with symmetrized initialization and report the average and standard deviation of the test loss over $10$ experiments. To ensure that the bad performances corresponding to large $\tau$ are not due to a lack of regularization, we display also the best test error throughout training (for kernel methods, early stopping is a form of regularization~\cite{yao2007early}).

\paragraph{Increasing number of parameters.} Figure~\ref{fig:losses}-(a) shows the evolution of the test error when increasing $m$ as discussed in Section~\ref{sec:range}, \emph{without} symmetrized initialization. We report the results for two choices of scaling functions $\alpha(m)$, averaged over $5$ experiments with $d=100$. The scaling $1/\sqrt{m}$ leads to lazy training, with a poor generalization as $m$ increases, in contrast to the scaling $1/m$ for which the test error remains relatively close to $0$ for large $m$ (more experiments with this scaling can be found in~\cite{chizat2018global,rotskoff2018neural,mei2018mean}).

\paragraph{Under-parameterized case.} Finally, Figure~\ref{fig:losses}-(b) illustrates the under-parameterized case, with $d=100, m=50$ with symmetrized initialization. We used SGD with batch-size 200 to minimize the population square loss, and displayed average and standard deviation of the final population loss (estimated with $2000$ samples) over $5$ experiments. As shown in Theorem~\ref{th:under}, SGD converges to a \emph{a priori} local minimum in the lazy regime (i.e., here for large $\tau$). In contrast, it behaves well when $\tau$ is small, as in Figure~\ref{fig:cover}. There is also an intermediate regime (hatched area) where convergence is very slow and the loss was still decreasing when the algorithm was stopped.

\subsection{Deep CNNs experiments}\label{sec:CNNexpe}
We now study whether lazy training is relevant to understand the good performances of convolutional neural networks (CNNs). 

\paragraph{Interpolating from standard to lazy training.}
We first study the effect of increasing the scale factor $\alpha$ on a standard pipeline for image classification on the CIFAR10 dataset.
We consider the VGG-11 model~\cite{simonyan2014very}, which is a widely used model on CIFAR10. We trained it via mini-batch SGD with a momentum parameter of $0.9$. For the sake of interpretability, no extra regularization (e.g., BatchNorm) is incorporated, since a simple framework that outperforms linear methods baselines with some margin is sufficient to our purpose (see Figure~\ref{fig:CNNs}(b)). An initial learning rate $\eta_0$ is linearly decayed at each epoch, following $\eta_t=\frac{\eta_0}{1+\beta t}$.  The biases are initialized with $0$ and all other weights are initialized with normal Xavier initialization~\cite{glorot2010understanding}. In order to set the initial output to $0$ we use the \emph{centered model} $h$, which consists in replacing the VGG model $\tilde h$ by $h(w) \coloneqq \tilde h(w)-\tilde h(w_0)$. Notice that this does not modify the differential at initialization.

The model $h$ is trained for the square loss multiplied by $1/\alpha^2$ (as in Section~\ref{sec:dynamics}), with standard data-augmentation, batch-size of 128 \cite{zagoruyko2016wide} and $\eta_0=1$ which gives the best test accuracies  over the grid $10^k$, $k\in \{-3,3\}$, for all $\alpha$. The total number of epochs is $70$, adjusted so that the performance reaches a plateau for $\alpha=1$. Figure \ref{fig:CNNs}(a) reports the accuracy after training $\alpha h$ for increasing values of $\alpha \in 10^k$ for $k=\{0,1,2,3,4,5,6,7\}$ ($\alpha=1$ being the standard setting).  For $\alpha<1$, the training loss diverges with $\eta_0=1$. We also report the \emph{stability of activations}, which is the share of neurons over ReLU layers that, after training, are activated for the same inputs than at initialization, see Appendix~\ref{app:experiments}. Values close to $100\%$ are strong indicators of an effective linearization.

 We observe a significant drop in performance as $\alpha$ grows, and then the accuracy reaches a plateau, suggesting that the CNN progressively reaches the lazy regime. This demonstrates that the linearized model (large $\alpha$) is not sufficient to explain the good performance of the model for $\alpha=1$. For large $\alpha$, we obtain a low limit training accuracy and do not observe overfitting, a surprising fact since this amounts to solving an over-parameterized linear system. This behavior is due to a poorly conditioned linearized model, see Appendix~\ref{app:experiments}.
 
 \paragraph{Performance of linearized CNNs.} In this second set of experiments, we investigate whether variations of the models trained above in a lazy regime could increase the performance and, in particular, could outperform other linear methods which also do not involve learning a representation~\cite{rahimi2008random,oyallon2015deep}.  To this end, we train widened CNNs in the lazy regime, as widening is a well-known strategy to boost performances of a given architecture~\cite{zagoruyko2016wide}. We multiply the number of channels of each layer by $8$ for the VGG model and $7$ for the ResNet model~\cite{he2016deep} (these values are limited by hardware constraints). We choose $\alpha=10^7$ to train the linearized models, a batch-size of 8 and,  after cross-validation, $\eta_0=0.01, 1.0$ for respectively the standard and  linearized model. We also multiply the initial weights by respectively $1.2$ and $1.3$ for the ResNet-18 and VGG-11, as we found that it slightly boosts the training accuracies. Each model is trained with the cross-entropy loss divided by $\alpha^2$ until the test accuracy stabilizes or increases, and we check that the average stability of activations (see Appendix~\ref{app:experiments}) was $100\%$.
 
As seen on Figure~\ref{fig:CNNs}(b), widening the VGG model slightly improves the performances of the linearized model compared to the previous experiment but there is still a substantial gap of performances from other non-learned representations \cite{recht2019imagenet,oyallon2015deep} methods, not to mention the even wider gap with their non-lazy counterparts. This behavior is also observed on the state-of-the-art ResNet architecture. Note that~\cite{arora2019exact} reports a test accuracy of $77.4\%$ without data augmentation for a linearized CNN with a specially designed architecture which in particular solves the issue of ill-conditioning. Whether variations of standard architectures and pipelines can lead to competitive performances with linearized CNNs, remains an open question.

\paragraph{Remark on wide NNs.} It was proved~\cite{jacot2018neural} that neural networks with standard initialization (random independent weights with zero mean and variance $O(1/n_\ell)$ at layer $\ell$, where $n_\ell$ is the size of the previous layer), are bound to reach the lazy regime as the sizes of all layers grow unbounded. Moreover, for very large neural networks of more than $2$ layers, this choice of initialization is essentially mandatory to avoid exploding or vanishing initial gradients~\cite{he2015delving,hanin2018start} if the weights are independent with zero mean. Thus we stress that we do not claim that wide neural networks do not show a lazy behavior, but rather that those which exhibit good performances are far from this asymptotic behavior.

\begin{figure}
\centering
\begin{subfigure}{0.4\linewidth}
\centering
\includegraphics[scale=0.6]{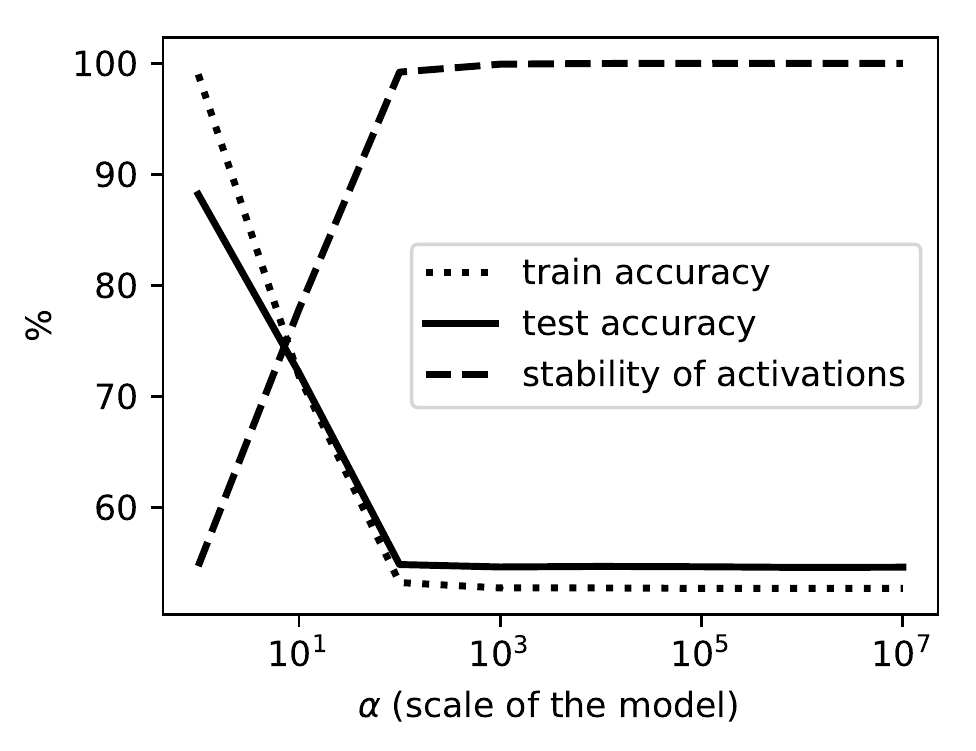}
\caption{}\label{fig:accvsalpha}
\end{subfigure}%
\begin{subfigure}{0.6\linewidth}
\centering
\begin{tabular}{lrr}
  \hline
  Model & Train acc. &  Test acc.\\
  \hline
 ResNet wide, linearized & 55.0 & 56.7\\
 VGG-11 wide, linearized & 61.0 & 61.7 \\
Prior features~\cite{oyallon2015deep}  &- & 82.3 \\
Random features~\cite{recht2019imagenet} & - & 84.2  \\
  VGG-11 wide, standard & 99.9  & 89.7 \\
 ResNet wide, standard & 99.4 & 91.0\\
  \hline
\end{tabular}
\caption{}\label{fig:perfwidelazy}
\end{subfigure}
\caption{(a) Accuracies on CIFAR10 as a function of the scaling $\alpha$. The stability of activations suggest a linearized regime when high. (b) Accuracies on CIFAR10 obtained for $\alpha=1$ (standard, non-linear) and $\alpha=10^7$ (linearized) compared to those reported for some linear methods without data augmentation: random features and prior features based on the scattering transform. }
\label{fig:CNNs}
\end{figure}

\section{Discussion}
Lazy training is an implicit bias phenomenon in differentiable programming, that refers to the situation when a non-linear parametric model behaves like a linear one.  This arises when the \emph{scale} of the model becomes large which, we have shown, happens implicitly under some choices of hyper-parameters governing normalization, initialization and number of iterations. While the lazy training regime provides some of the first optimization-related theoretical insights for deep neural networks~\cite{du2018gradientdeep, allenzhu2018convergence, zou2018sgdoptimizes,jacot2018neural}, we believe that it does not explain yet the many successes of neural networks that have been observed in various challenging, high-dimensional  tasks in machine learning. This is corroborated by numerical experiments where it is seen that the performance of networks trained in the lazy regime degrades and in particular does not exceed that of some classical linear methods. Instead, the intriguing phenomenon that still defies theoretical understanding is the one illustrated on Figure~\ref{fig:cover}(c) for small $\tau$ and on Figure~\ref{fig:CNNs}(a) for $\alpha=1$: neural networks trained with gradient-based methods (and neurons that move) have the ability to perform high-dimensional feature selection through highly non-linear dynamics.

\subsubsection*{Acknowledgments}
We acknowledge supports from grants from R\'egion Ile-de-France and the European Research Council (grant SEQUOIA 724063). Edouard Oyallon was supported by a GPU donation from NVIDIA. We thank Alberto Bietti for interesting discussions and Brett Bernstein for noticing an error in a previous version of this paper.

\bibliography{LC}
\bibliographystyle{plain}

\newpage

\section*{Supplementary material}
Supplementary material for the paper: ``On Lazy Training in Differentiable Programming'' authored by L\'ena\"ic Chizat , Edouard Oyallon and Francis Bach (NeurIPS 2019). This supplementary material is organized as follows:
\begin{itemize}
\item Appendix~\ref{app:tangent}: Remarks on the linearized model
\item Appendix~\ref{app:proofs}: Proofs of the theoretical results
\item Appendix~\ref{app:experiments}: Experimental details and additional results
\end{itemize}

\appendix


\section{The linearized model in supervised machine learning}\label{app:tangent}
\label{sec:tangent}

\subsection{Differentiable models and their linearization}
In this section, we give some details on the interpretation of the linearized model in the case of supervised machine learning. In this setting, a differentiable model is a typically a function  $f:\RR^p\times \RR^d \to \RR^k$ where $\RR^p$ is the parameter space, $\RR^d$ is the input space and $\RR^k$ the output space. One defines a Hilbert space $\Ff$ of functions from $\RR^d$ to $\RR^k$, typically $L^2(\rho_x,\RR^k)$ where $\rho_x$ is the distribution of input samples. The function $h:\RR^p\to \Ff$ considered in the article is then the function which to a vector of parameters associates a predictor $h : w\mapsto f(w,\cdot)$.

In first order approximation around the initial parameters $w_0\in \RR^p$, the parametric model $f(w,x)$ reduces to the following \emph{linearized} or \emph{tangent} model :
\begin{equation}\label{eq:tangent}
\bar f(w,x) = f(w_0,x) + D_w f(w_0,x)(w-w_0).
\end{equation}
where $D_w f$ is the differential of $f$ in the variable $w$.
The corresponding hypothesis class is \emph{affine} in the space of predictors.  It should be stressed that when $f$ is a neural network, $\bar f$ is generally not a linear neural network because it is not linear in $x\in \RR^d$, but in the features $D_w f(w_0,x) \in \RR^{p\times k}$ which generally depend non-linearly on $x$. For large neural networks, the dimension of the features might be much larger than $d$, which makes~$\bar f$ similar to a non-parametric method. Finally, if $f$ is already a linear model, then $f$ and~$\bar f$ are identical.

\paragraph{Kernel method with an offset.} In the case of the square loss, training the affine model~\eqref{eq:tangent} is equivalent to training a linear model in the variables
\[
(\tilde x, \tilde y) \coloneqq(D_w f(w_0,x),y-f(w_0,x)).
\] 
When $k=1$, this is equivalent to a kernel method with the \emph{tangent kernel}~\cite{jacot2018neural} defined as $K:\RR^d\times \RR^d \to \RR$
\begin{equation}\label{eq:tangentkernel}
K(x,x') = D_w f(w_0,x) D_w f(w_0,x')^\intercal.
\end{equation}
This kernel is different from the one traditionally associated to neural networks~\cite{rahimi2009weighted,daniely2016toward} which involve the derivative with respect to the output layer only.
Also, the output data is shifted by the initialization of the model $h(w_0)=f(w_0,\cdot)$. This term inherits from the randomness due to the initialization: it is for instance shown in~\cite{lee2017deep,matthews2018gaussian} that the distribution of $h(w_0)$ converges to a Gaussian process for certain over-parameterized neural networks initialized with random normal weights.

 \subsection{Two-layer neural networks}
 \label{app:randomfeature}
Lazy training has some interesting consequences when looking more particularly at two-layer neural networks. These are functions of the form 
\[
f_m(w,x) = \alpha(m) \sum_{j=1}^m b_{j}\cdot  \sigma(a_j \cdot x),
\]
where $m\in \mathbb{N}$ is the size of the hidden layer and $\sigma: \RR\to \RR$ is an activation function and
the parameters\footnote{We have omitted the bias/intercept, which is recovered by fixing the last coordinate of $x$ to $1$.} are $(\theta_j)_{j=1}^m$ where $\theta_j = (a_j,b_j) \in \RR^{d+1}$, so here the number of parameters is $p=m(d+1)$. We have also introduced a scaling $\alpha(m)>0$ as in Section~\ref{sec:range}.
 
\textbf{Justification for asymptotics.} In this paragraph, we justify the formula for the asymptotic upper bound on $\kappa_{h_m}(w_0)$ given for such models in Section~\ref{sec:range}. Using the assumption that $\EE\phi(\theta_i)=0$ and the fact that the parameters are independents, one has $\EE \Vert h(w_0)\Vert^2 = m \alpha(m)^2 \EE \Vert\phi(\theta)\Vert^2$.
For the differential, from the law of large numbers, we have the estimate
\[
\frac1{m\alpha(m)^2} Dh(w_0)\,Dh(w_0)^\intercal = \frac{1}m \sum_{i=1}^m D\phi(\theta_i)D\phi(\theta_i)^\intercal \underset{m\to \infty}{\longrightarrow} \EE\left[D\phi(\theta)D\phi(\theta)^\intercal\right].
\]
It follows that $\EE \Vert Dh(w_0) \Vert^2 = \EE \Vert Dh(w_0)Dh(w_0)^\intercal\Vert \sim m\alpha(m)^2 \Vert \EE[D\phi(\theta)D\phi(\theta)^\intercal]\Vert$ because we have assumed that $D\phi$ is not identically $0$ on the support of $\theta$. One also has
\[
\Vert D^2h(w_0) \Vert = \sup_{\substack{u\in \RR^{d\times m}\\ \Vert u\Vert \leq 1}} \alpha(m) \sum_{i=1}^m u_i^\intercal D^2\phi(\theta_i)u_i \leq   \alpha(m)  \sup_{\theta_i} \Vert D^2\phi(\theta_i)\Vert \leq \alpha(m)\Lip(D\phi).
\]
From the definition of $\kappa_{h_m}(w_0)$ and the upper bound $\Vert h_m(w_0)-y^\star\Vert \leq \Vert h(w_0)\Vert+ \Vert y^\star \Vert$ we conclude that
\[
 \EE[\kappa_{h_m}(w_0)] \lesssim m^{-\frac12} +(m\alpha(m))^{-1} .
\]

\textbf{Limit kernels and random feature.} 
In this section, we show that the tangent kernel is a \emph{random feature kernel} for neural networks with a single hidden layer. For simplicity, we consider the scaling $\alpha(m)=1/\sqrt{m}$ as in~\cite{du2018gradient} which leads to a non-degenerated limit of the kernel\footnote{Since the definition of gradients depends on the choice of a metric, this scaling is not of intrinsic importance. Rather, it reflects that we work with the Euclidean metric on $\RR^p$. The choice of scaling however becomes important when dealing with training (see also discussion in Section~\ref{sec:range}).} as $m\to \infty$. The associated tangent kernel in Eq.~\eqref{eq:tangentkernel} is the sum of two kernels $K_m(x,x') = K_m^{(a)}(x,x') + K_m^{(b)}(x,x')$, one for each layer, where
\begin{align*}
K_m^{(a)}(x,x') = \frac1m \sum_{j=1}^m (x \cdot x' ) b_j^2\sigma'(a_j \cdot x)\sigma'(a_j \cdot x') &&\text{and}&&
K_m^{(b)}(x,x') = \frac1m \sum_{j=1}^m \sigma(a_j \cdot x)\sigma(a_j \cdot x').
\end{align*}
If we assume that the initial weights $a_j$ (resp.~$b_j$) are independent samples of a distribution on $\RR^{d}$ (resp. a distribution on $\RR$), these are random feature kernels~\cite{rahimi2008random} that converge as $m\to \infty$ to the kernels
\begin{align*}
K^{(a)}(x,x') =  \mathbb{E}_{(a,b)} \left[ (x\cdot x') b^2 \sigma'(a\cdot x)\sigma'(a\cdot x') \right] &&\text{and}&&
K^{(b)}(x,x') =  \mathbb{E}_{a} \left[ \sigma(a\cdot x)\sigma(a\cdot x') \right].
\end{align*}
The second component $K^{(b)}$, corresponding to the differential with respect to the output layer, is the one traditionally used to make the link between these networks and random features~\cite{rahimi2009weighted}. When $\sigma(s)=\max\{s,0\}$ is the rectified linear unit activation and the distribution of the weights~$a_j$ is rotation invariant in $\RR^d$, one has the following explicit formulae~\cite{cho2009kernel}:
\begin{align}\label{eq:explicitkernel}
K^{(a)}(x,x')\! =\!  \frac{(x\cdot x')\EE(b^2)}{2\pi} (\pi-\varphi),
 &&
K^{(b)}(x,x') \!= \! \frac{\Vert x\Vert \Vert x'\Vert \EE(\Vert a\Vert^2 )}{2\pi d}((\pi-\varphi)\cos \varphi+ \sin \varphi)
\end{align}
where $\varphi\in [0,\pi]$ is the angle between the two vectors $x$ and $x'$. See Figure~\ref{fig:randomkernel} for an illustration of this kernel and the convergence of its random approximations.
The link with (independent) random sampling is lost for deeper neural networks, but it is shown in~\cite{jacot2018neural} that tangent kernels still converge when the size of networks increase, for certain architectures.

\begin{figure}
\centering
\includegraphics[scale=0.4]{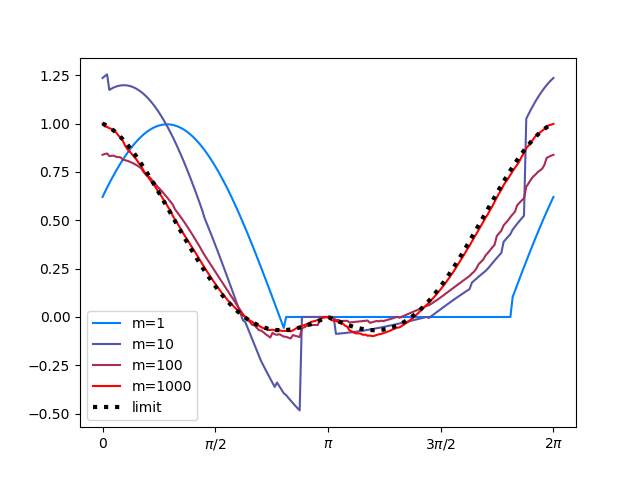}


\caption{Random realizations of the kernels $K_m$ and the limit kernel $K$ of Eq.~\eqref{eq:explicitkernel}. We display the value of $K(x,x')$ as a function of $\varphi=\mathrm{angle}(x,x')$ with $x$ fixed, on a section of the sphere in $\RR^{10}$. Parameters are normal random variables of variance $1$, so $\EE(b^2)=1$ and $\EE(\Vert a\Vert^2)=d$.}\label{fig:randomkernel}
\end{figure}

\subsection{Generalization for the lazy model}\label{subsec:generalization}
As noted in the main text, in supervised machine learning, $\Ff$ is often a Hilbert space of functions on $\RR^d$ and the model $h$ is often of the form $h(w)=f(w,\cdot)$ where $f:\RR^p\times \RR^d\to \RR^k$. A natural question that arises in this context and that is not directly answered by the theorems of Section~\ref{sec:dynamics}, is whether the trained lazy model and the trained tangent model also generalize the same way, i.e.\ whether at training time $T$, it holds $ f(w(T),x) \approx \bar f(\bar w(T),x)$ for points $x\in\RR^d$ that are not in the training set, where $\bar f(w,x)=f(w_0,x)+D_w f(w_0,x)(w-w_0)$. We will see here that it is actually a simple consequence of the bounds.
\begin{proposition}[Generalizing like the tangent model]
Assume that for some $C>0$ it holds $\Vert w_\alpha(T) - \bar w(T)\Vert \leq C\log(\alpha)/\alpha^2$. 
Assume moreover that there exists a set $\Xx\subset \RR^d$ such that $M_1 \coloneqq \sup_{x\in \Xx} \Vert D_w f(w_0,x)\Vert <\infty$ and $M_2 \coloneqq \sup_{x\in \Xx} \Lip(w\mapsto D_w f(w,x)) <\infty$. Then it holds
\[
\sup_{x\in \Xx} \Vert \alpha f(w_\alpha(T),x)- \alpha \bar f(\bar w_\alpha(T),x)\Vert \leq C\frac{\log\alpha}{\alpha}\left(M_1 +\frac12 C\cdot M_2\cdot \log(\alpha)\right) \underset{\alpha\to \infty} {\longrightarrow}0.
\]
\end{proposition}
\begin{proof}
Let us call $A$ the quantity to be upper bounded, and start with the decomposition
\[
A \leq \sup_{x\in \Xx} \Vert \alpha f(w_\alpha(T),x)- \alpha \bar f(w_\alpha(T),x)\Vert  + \sup_{x\in \Xx} \Vert \alpha \bar f(w_\alpha(T),x)- \alpha \bar f(\bar w_\alpha(T),x)\Vert = A_1 + A_2
\]
By Taylor's theorem applied at each point $x\in X$, one has
\[
A_1 \leq \frac\alpha2 M_2 \Vert w_\alpha(T) -\bar w_\alpha(T)\Vert^2 \leq \frac{C^2\cdot M_2\log(\alpha)^2}{2\alpha}.
\]
It also holds 
\[
A_2 = \alpha \sup_{x\in \Xx} \Vert D_w f(w_0,x)(w_\alpha(T)-\bar w_\alpha(T)) \Vert \leq \frac{M_1C\log(\alpha)}{\alpha}
\]
and the conclusion follows.
\end{proof}


\section{Proofs of the theoretical results}\label{app:proofs}
In all the forthcoming proofs, we use the notations $y(t)=\alpha h(w_\alpha(t))$ and $\bar y(t) = \alpha \bar h(\bar w_\alpha(t))$ for the dynamics in $\Ff$ (they also depend on $\alpha$ although this is not reflected in the notation).  We also write $\Sigma(w) \coloneqq Dh(w)Dh(w)^\intercal$ for the so-called tangent kernel~\cite{jacot2018neural}, which is a quadratic form on $\Ff$. By using the chain rule, we find that the trajectories in $\Ff$ solve the differential equation 
\begin{align*}
\frac{d}{dt} y(t) &= - \Sigma(w_\alpha(t))\nabla R(y(t)), \\
\frac{d}{dt} \bar y(t) &= - \Sigma(w(0))\nabla R(\bar y(t)) .
\end{align*}
with $y(0)=\bar y(0)=\alpha h(w_0)$. Remark that the first differential equation is coupled with $w_\alpha(t)$.

\subsection{Proof for Theorem~\ref{th:generallazy} (finite horizon, non-quantitative)}
For this first proof, we only track the dependency in $\alpha$, and we use $C$ to denote a quantity independent of $\alpha$, that may vary from line to line. For $T>0$, it holds
\begin{align*}
\int_0^T \Vert w'_\alpha(t)\Vert dt  = \int_0^T \Vert \nabla F_\alpha(w_\alpha(t))\Vert dt \leq \sqrt{T}\left(\int_0^T \Vert \nabla F_\alpha(w_\alpha(t))\Vert^2 dt \right)^{\frac12}.
\end{align*}
It follows, by using the fact that $\frac{d}{dt} F_\alpha(w_\alpha(t)) = -\Vert \nabla F_\alpha(w_\alpha(t))\Vert^2 $, that $\sup_{t\in [0,T]} \Vert w_\alpha(t) - w(0)\Vert \leq \left(T\cdot F_\alpha(w_\alpha(t))\right)^{\frac12} \lesssim \frac1\alpha$. In particular, we deduce that $\sup_{t\in [0,T]}\Vert y(t)-y(0)\Vert \leq C$ and $\sup_{t\in [0,T]} \Vert\nabla R(y(t))\Vert\leq C$.

Let us now consider the evolution of $\Delta(t)\coloneqq \Vert y(t)-\bar y(t)\Vert$. It satisfies $\Delta(0)=0$ and 
\begin{align*}
\Delta'(t) &\leq \Vert \Sigma(w_\alpha(t)) \nabla R(y(t)) - \Sigma(w(0)) \nabla R(\bar y(t))\Vert \\
&\leq \Vert (\Sigma(w_\alpha(t)) - \Sigma(w(0))) \nabla R(y(t))\Vert + \Vert \Sigma(w(0)) (\nabla R(y(t)) -  \nabla R(\bar y(t))\Vert\\
&\leq C_1/\alpha + C_2\Delta(t)
\end{align*}
The ordinary differential equation $u'(t) = C_1/\alpha +C_2 u(t)$ with initial condition $u(0)=0$ admits the unique solution $u(t) = \frac{C_1}{\alpha C_2}(\exp(C_2 t)-1)$. Since $\Delta(t)$ is a sub-solution of this system, it follows that $\Delta(t)\leq \frac{C_1}{\alpha C_2}(\exp(C_2 t)-1) \leq C/\alpha$ (notice the exponential dependence in the final time and some other characteristics of the problem). 
Finally, consider the quantity $\delta(t) = \Vert w_\alpha(t)- \bar w_\alpha(t)\Vert$. It holds
\begin{align*}
\delta'(t)&\leq \alpha^{-1}  \Vert Dh(w_\alpha(t))^\intercal \nabla R(f(t)) - Dh(w_0)^\intercal \nabla R(\bar y(t)) \Vert \\
&\leq \alpha^{-1}   \Vert Dh(w_\alpha(t))^\intercal -Dh(w_0)^\intercal \Vert  \Vert \nabla R(y(t))\Vert +  \alpha^{-1} \Vert D h(w_0)\Vert \Vert \nabla R(y) - \nabla R(\bar y(t))\Vert\\
&\leq C \alpha^{-2}
\end{align*}
We thus conclude, since $\delta(0)=0$, that $\sup_{t\in [0,T]} \Vert \delta(t)\Vert \leq \alpha^{-2}$.

\subsection{Proof of Theorem~\ref{th:square} (finite horizon, square loss)}\label{app:squareproof}
\textbf{Step 1.} With the square loss, the objective is still potentially non-convex, but we have the property
\begin{equation*}\label{eq:squareproperty}
\frac{d}{dt} \Vert y(t) - y^*\Vert^2 = - \langle \Sigma(w(t))(y(t)-y^\star),y(t)-y^\star\rangle \leq 0.
\end{equation*}
The proof scheme is otherwise similar as above, but we carry all constants. Let us denote $T_{exit} = \inf\{t>0\;;\; \Vert w_\alpha(t) - w_0\Vert>r\}$. For $t\leq T_{exit}$ it holds
\[
\Vert w'_\alpha(t)\Vert = \Vert \nabla F_\alpha(w_\alpha(t))\Vert \leq \alpha^{-1}\Vert y(t)-y^\star\Vert  \Vert Dh(w_\alpha(t))\Vert \leq \alpha^{-1} \Vert y(0)-y^\star\Vert \Lip(h)
\]
It follows that $\Vert w_\alpha(t)-w(0)\Vert \leq t\alpha^{-1} \Vert y(0)-y^\star\Vert \Lip(h)$ (this bound is tighter for small times, compared to the bound in $\sqrt{t}$ used in the previous proof). Since we have assumed that 
 $\alpha\geq k \Vert y(0)-y^\star\Vert/(r\Lip(h))$, it holds $\Vert w_\alpha(t) -w_0\Vert \leq (t/K)\cdot r\Lip(h)^2 =r$ so $T_{exit}>T$.

\textbf{Step 2.} Now we consider $\Delta(t) = \Vert y(t)-\bar y(t)\Vert$. It holds 
\begin{align*}
\frac12 \frac{d}{dt} \Delta(t)^2 
&= \langle y'(t) -\bar y'(t),y(t) -\bar y(t)\rangle \\ 
 &\leq - \langle \Sigma(w_\alpha(t)) \nabla R(y(t)) - \Sigma(w(0)) \nabla R(\bar y(t)), y(t) - \bar y(t)\rangle \\
&\leq - \langle (\Sigma(w_\alpha(t)) - \Sigma(w(0))) \nabla R(y(t)), y(t) -\bar y(t)\rangle
\end{align*}
where we have used the fact that $\langle \Sigma(w(0)) (\nabla R(y(t)) -  \nabla R(\bar y(t)), y(t) -\bar y(t)\rangle\geq 0$, which is specific to the square loss. Taking the norms and dividing both sides by $\Delta(t)$, it follows
\[
\Delta'(t) \leq \Lip(\Sigma)\cdot \Vert w_\alpha(t) -w(0)\Vert \Vert y(0)-y^\star\Vert \leq 2 \Lip(h)^2\Lip(Dh) t\alpha^{-1} \Vert y(0)-y^\star\Vert^2
\]
where we have used $\Lip(\Sigma) \leq 2\Lip(h)\Lip(Dh)$. Since $\Delta(0)=0$, it follows
\[
\Delta(t) \leq \frac{t^2}{\alpha}  \Lip(h)^2\Lip(Dh) \Vert y(0) -y^\star\Vert^2.
\]
The bound in the statement then follows by writing this upper bound at time $T=K/\Lip(h)^2$.

\textbf{Step 3.} Finally, consider $\delta(t) = \Vert w_\alpha(t)- \bar w_\alpha(t)\Vert$. The bound that we will obtain is not reported in the main text due to space constraints, but proved here for the sake of completeness. As in the previous proof, it holds
\[
\alpha \delta'(t) \leq   \Vert Dh(w_\alpha(t))^\intercal -Dh(w_0)^\intercal \Vert  \Vert \nabla R(y(t))\Vert + \Vert D h(w_0)\Vert \Vert \nabla R(y) - \nabla R(\bar y(t))\Vert = A(t) +B(t).
\]
Let us bound these two quantities separately. On the one hand, it holds for $t\in [0,T]$,
\[
A(t) \leq  \Lip(Dh) \Vert w_\alpha(t) -w_0\Vert \Vert y(0)-y^\star\Vert \leq \frac{t}{\alpha}\Lip(h)\Lip(Dh) \Vert y(0)-y^\star\Vert^2.
\]
On the other hand, it holds for $t\in [0,T]$,
\[
B(t) \leq \frac{t^2}{\alpha} \Lip(h)^3 \Lip(Dh)\Vert y(0)-y^\star\Vert^2.
\]
By integrating these two bounds and summing, we get 
\begin{align*}
\delta(T) &\leq \frac{T^2}{\alpha^2} \Lip(h)^2\Lip(Dh)\Vert y(0)-y^\star\Vert^2\left(\frac{2}{\Lip(h)}+\frac{4T}{3}\Lip(h)\right)\\
&\leq \frac{K^2}{\alpha^2}\frac{\Lip(Dh)}{\Lip(h)^3}\Vert y(0)-y^\star\Vert^2\left(2 +4K/3\right).
\end{align*}
After rearranging the terms, we obtain
\[
\frac{\alpha \Lip(h)}{\Vert y(0)-y^\star\Vert} \Vert w_\alpha(T)-\bar w_\alpha(T)\Vert \leq \frac{K^2}{\alpha} \frac{\Lip(Dh)}{\Lip(h)^2}\Vert y(0)-y^\star\Vert \left( 2+4K/3 \right)
\]
Note that this bound is arranged so that both sides of the inequality are dimensionless, in the sense that they would not change under a simple rescaling of either the norm on $\Ff$ or on $\RR^p$. The left-hand side should be understood as the relative difference between the non-linear and the linearized dynamics, while the right-hand side involves the \emph{scale} of Section~\ref{sec:range}.

\subsection{Proof of Theorem~\ref{th:over} (over-parameterized case)}\label{app:overparam}
Consider the radius $r_0 \coloneqq \sigma_{\min}/(2\Lip(Dh))$. By smoothness of $h$, it holds $\Sigma(w)\succeq \sigma_{\min}^2 \Id/4$ as long as $\Vert w-w_0\Vert< r_0$. Thus Lemma~\ref{lem:stronglyconvexGF} below guarantees that $y(t)$ converges linearly, up to time $T \coloneqq \inf \{ t\geq 0\; ; \; \Vert w_\alpha(t)-w_0\Vert > r_0\}$. It only remains to find conditions on $\alpha$ so that $T=+\infty$.
The variation of the parameters $w_\alpha(t)$ can be bounded for $0\leq t \leq T$ as
\[
\Vert w'_\alpha(t)\Vert \leq \frac1\alpha \Vert Dh(w_\alpha(t))\Vert \Vert\nabla R(y(t))\Vert \leq \frac{2 M}{\alpha} \Vert Dh(w_0)\Vert \Vert y(t)- y^*\Vert.
\]
By Lemma~\ref{lem:stronglyconvexGF}, it follows that for $0\leq t\leq T$, 
\begin{align*}
\Vert w_\alpha(t) - w_0\Vert  &\leq \frac{2M^{3/2}}{\alpha m} \Vert Dh(w_0)\Vert \Vert y(0)- y^*\Vert \int_0^t e^{-(m \sigma_{\min}^2/4) s}ds \\
&\leq \frac{8\kappa^{3/2} }{\alpha \sigma_{\min}^2} \Vert Dh(w_0)\Vert \Vert y(0)-y^*\Vert.
\end{align*}
This quantity is smaller than $r_0$, and thus $T=\infty$, if $\Vert y(0) -y^*\Vert \leq 2\alpha C_0$. This is in particular guaranteed by the conditions on $h(w_0)$ and $\alpha$ in the theorem. 

When $h(w_0)=0$, the previous bound also implies the ``laziness'' property $\sup_{t\geq 0} \Vert w_\alpha(t)-w_0\Vert = O(1/\alpha)$ since in that case $y(0)$ does not depend on $\alpha$. For the comparison with the tangent gradient flow, the first bound is obtained by applying the stability Lemma~\ref{lem:stability}, and noticing that the quantity denoted by $K$ in that lemma is in $O(1/\alpha)$ thanks to the previous bound on $\Vert w_\alpha(t)-w_0\Vert$. For the last bound, we compute the integral over $[0,+\infty)$ of the bound
\begin{align*}
\alpha \Vert w'_\alpha(t)-\bar w'_\alpha(t) \Vert 
&= \Vert Dh(w_\alpha(t))^\intercal \nabla R(y(t)) - Dh(w_0)^\intercal \nabla R(\bar y(t)) \Vert \\
&\leq \Vert Dh(w_\alpha(t)) -Dh(w_0)\Vert  \Vert \nabla R(y(t))\Vert + \Vert Dh(w_0)\Vert \Vert \nabla R(y(t)) - \nabla R(\bar y(t))\Vert.
\end{align*}
It is easy to see from the derivations above that the integral of the first term is in $O(1/\alpha)$. For the second term, we define $t_0\coloneqq 4\log \alpha/(\mu \sigma_{\min}^2)$ and on $[0,t_0]$ we use the smoothness bound
\[
\Vert \nabla R(y(t)) - \nabla R(\bar y(t))\Vert \leq M \Vert y(t) - \bar y(t) \Vert
\]
which integral over $[0,t_0]$ is in $O(\log \alpha/\alpha)$, while on $[t_0,+\infty)$ we use the crude bound
\[
\Vert \nabla R(y(t)) - \nabla R(\bar y(t))\Vert  \leq \Vert \nabla R(y(t))\Vert + \Vert \nabla R(\bar y(t)) \Vert
\]
which integral over $[t_0,+\infty)$ is in $O(1/\alpha)$ thanks to the definition of $t_0$ and the exponential decrease of $\nabla R$ along both trajectories. This is sufficient to conclude. As a side note, we remark that the assumption that $Dh$ is globally Lipschitz could be avoided by considering the more technical definition
\[
\Lip(Dh) \coloneqq \inf\left\{ L>0 \;;\; Dh \text{ is $L$-Lipschitz on a ball centered at $w_0$ of radius $\frac{\sigma_{\min}}{2L}$}\right\}>0,
\]
because then the path $w_\alpha(t)$ never escapes the ball of radius $\frac{\sigma_{\min}}{2L}$ around $w_0$ for $\alpha >\Vert y^*\Vert/C_0$.

\begin{lemma}[Strongly-convex gradient flow in a time-dependent metric]\label{lem:stronglyconvexGF}
Let $F: \Ff \to \RR$ be a $m$-strongly-convex function with $M$-Lipschitz continuous gradient and with global minimizer $y^*$ and let $\Sigma(t):\Ff \to \Ff$ be a time dependent continuous self-adjoint linear operator with eigenvalues lower bounded by $\lambda>0$ for $0\leq t \leq T$. Then solutions on $[0,T]$ to the differential equation
\begin{equation*}\label{eq:modelgradientflow}
y'(t) =  - \Sigma(t) \nabla F(y(t)),
\end{equation*}
 satisfy, for $0\leq t\leq T$,
\begin{equation*}\label{eq:linearconvergence}
\Vert y(t)-y^*\Vert \leq (M/m)^{1/2}\Vert y(0) - y^*\Vert \exp\left( -m \lambda t\right).
\end{equation*}
\end{lemma}
\begin{proof} By strong convexity, it holds $\bar F(y) \coloneqq F(y)-F(y^*) \leq \frac1{2m}\Vert \nabla F(y)\Vert^2$. It follows
\begin{align*}
\frac{d}{dt} \bar F(y(t))  = - \nabla F(y(t))^\intercal \, \Sigma(t) \, \nabla F(y(t)) \leq -\lambda \Vert \nabla F(y(t))\Vert^2 \leq -2m \lambda \bar F(y) ,
\end{align*}
and thus $\bar F(y(t))  \leq \exp\left( -2m \lambda \right) \bar F(y(0))$ by Gr\"onwall's Lemma. We now use the strong convexity inequality $\Vert y-y^*\Vert^2 \leq \frac2{m} \bar F(y)$ in the left-hand side and the smoothness inequality $\bar F(y) \leq \frac12 M \Vert y - y^*\Vert^2$ in the right-hand side. This yields $\Vert y(t)-y^*\Vert^2 \leq \frac{M}{m}  \exp\left( -2m\lambda \right) \Vert y(0) - y^*\Vert^2$.
\end{proof}

\subsection{Stability Lemma}\label{app:stability}
The following stability lemma is at the basis of the equivalence between lazy training and linearized model training in Theorem~\ref{th:over}. We limit ourselves to a rough estimate sufficient for our purposes.
\begin{lemma}\label{lem:stability}
Let $R: \Ff\to \RR_+$ be a $m$-strongly convex function and let $\Sigma(t)$ be a time dependent positive definite operator on $\Ff$ such that $\Sigma(t)\succeq \lambda \Id$ for $t\geq 0$. Consider the paths $y(t)$ and $\bar y(t)$ on $\Ff$ that solve for $t\geq 0$,
\begin{align*}
y'(t) = -\Sigma(t) \nabla R(y(t)) && \text{and} && \bar y'(t) = - \Sigma(0) \nabla R(\bar y(t)).
\end{align*}
Defining $K\coloneqq \sup_{t\geq 0} \Vert  (\Sigma(t) -\Sigma(0))\nabla R(y(t))\Vert $, it holds for $t\geq 0$,
\[
\Vert y(t) - \bar y(t) \Vert \leq \frac{K\Vert \Sigma(0)\Vert^{1/2}}{\lambda^{3/2} m}.
\]
\end{lemma}

\begin{proof}
Let $\Sigma_0^{1/2}$ be the positive definite square root of $\Sigma(0)$, let $z(t)=\Sigma^{-1/2}_0 y(t)$, $\bar z(t) =\Sigma_0^{-1/2} \bar y(t)$ and let $h:\RR_+\to \RR_+$ be the function defined as $h(t)=\frac12 \Vert z(t)-\bar z(t)\Vert^2$. It holds
\begin{align*}
h'(t) &= \langle z'(t) - \bar z'(t), z(t) - \bar z(t)\rangle \\
& = - \langle \Sigma^{-1/2}_0\Sigma(t)\nabla R(\Sigma^{1/2}_0 z(t)) -\Sigma_0^{1/2}\nabla R(\Sigma^{1/2}_0 \bar z(t)), z(t)-\bar z(t)\rangle\\
& = - \langle \Sigma^{1/2}_0 \nabla R(\Sigma^{1/2}_0 z(t)) - \Sigma^{1/2}_0 \nabla R(\Sigma^{1/2}_0 \bar z(t)), z(t) -\bar z(t)\rangle \tag{$A(t)$}\\
&\quad\, - \langle \Sigma^{-1/2}_0(\Sigma(t) -\Sigma(0)) \nabla R(\Sigma^{1/2}_0 z(t)) ,z(t)-\bar z(t)\rangle. \tag{$B(t)$}
\end{align*}
Since the function $z\mapsto R(\Sigma^{1/2}_0 z)$ is $\lambda m$-strongly convex, one has that $A(t)\leq - 2\lambda m h(t)$.
Using the quantity $K$ introduced in the statement, one has also $\Vert B(t)\Vert \leq  K \Vert z(t) - \bar z(t)\Vert/\sqrt{\lambda} = K\sqrt{2h(t)/\lambda}$.  Summing these two terms yields the bound
\[
h'(t) \leq K\sqrt{2h(t)/\lambda} - 2\lambda m h(t). 
\]
The right-hand side is a concave function of $h(t)$ which is nonnegative for $h(t) \in [0, K^2/(2\lambda^3 m^2)]$ and negative for higher values of $h(t)$. Since $h(0)=0$ it follows that for all $t\geq 0$, one has $h(t)\leq K^2/(2\lambda^3\mu^2)$ and the result follows since $\Vert y(t)-\bar y(t)\Vert\leq \Vert \Sigma(0)\Vert^{1/2} \sqrt{2h(t)}$.
\end{proof}

\subsection{Proof of Theorem~\ref{th:under} (under-parameterized case)}\label{app:underparam}
The setting of this theorem is depicted on Figure~\ref{fig:manifold}. By the rank theorem (a result of differential geometry, see~\cite[Thm.\ 4.12]{lee2003smooth} or~\cite{abraham2012manifolds} for a statement in separable Hilbert spaces), there exists open sets $\Ww_0, \bar \Ww_0  \subset \RR^p$ and $\Ff_0, \bar \Ff_0 \subset \Ff$ and diffeomorphisms $\varphi: \Ww_0 \to \bar \Ww_0$ and $\psi: \Ff_0 \to \bar \Ff_0$ such that $\varphi(w_0)=0$, $\psi(h(w_0))=0$ and $\psi \circ h \circ \varphi^{-1}=\pi_r$, where $\pi_r$ is the map that writes, in suitable bases, $(x_1,\dots,x_p)\mapsto (x_1,\dots,x_r,0,\dots)$. Up to restricting these domains, we may assume that $\bar \Ff_0$ is convex. We also denote by $\Pi_r$ the $r$-dimensional hyperplan in $\Ff$ that is spanned by the first $r$ vectors of the basis. The situation is is summarized in the following commutative diagram:
\begin{center}
\begin{tikzpicture}
  \matrix (m) [matrix of math nodes,row sep=3em,column sep=4em,minimum width=2em]
  {
     \mathcal{W}_0  & \mathcal{F}_0 \\
    \bar{\mathcal{W}}_0  & \bar{\mathcal{F}}_0 \\};
  \path[-stealth]
    (m-1-1) edge node [left] {$\varphi$} (m-2-1) edge node [above] {$h$} (m-1-2)
    (m-2-1.east|-m-2-2) edge node [below] {$\pi_r$} (m-2-2)
    (m-1-2) edge node [right] {$\psi$} (m-2-2);
\end{tikzpicture}
\end{center}
In the rest of the proof, we denote by $C>0$ any quantity that depends on $m$, $M$ and Lipschitz smoothness constants of $h,\psi,\varphi,\psi^{-1},\varphi^{-1}$, but not on $\alpha$. Although we do not do so, this could be translated into explicit constants that depends on the smoothness of $h$ and $R$, on the strong convexity constant of $R$ and on the smallest positive singular value of $Dh(w_0)$ using quantitative versions of the rank theorem~\cite[Thm. 2.7]{boutaib2015lipschitz}.

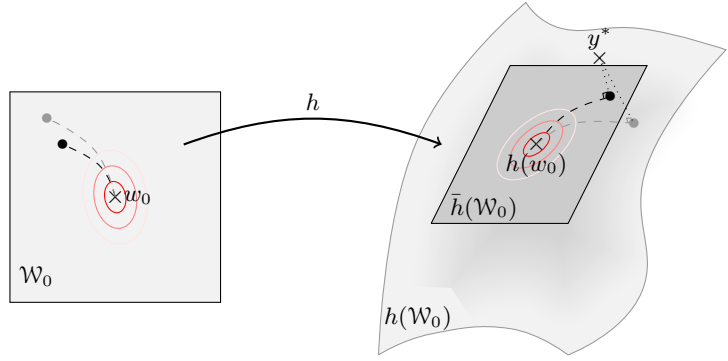
\begin{figure}
\centering
\small
\begin{tikzpicture}[scale=0.7]	
	\draw[fill=gray!10] (-10.0,-2.0) rectangle (-6.0,2);
	\coordinate (T) at (-8,0);
	\coordinate (T1) at (-9,1);
	\coordinate (T2) at (-9.3,1.5);
		\draw (T) node {$\times$};
		\draw (T1) node {$\bullet$};
		\draw[gray!80] (T2) node {$\bullet$};
	\draw[rotate=-80,color=red!100] (T) ellipse (0.3cm and 0.2cm);
	\draw[rotate=-80,color=red!50] (T) ellipse (0.6cm and 0.4cm);
	\draw[rotate=-80,color=red!10] (T) ellipse (0.9cm and 0.6cm);
		\draw[right] (T) node {$w_0$};
		\draw (-9.5,-1.5) node {$\mathcal{W}_0$};
		
		\draw [tension=0.80,dashed] plot [smooth] coordinates { (T) (-8.3,0.6) (T1)};
	\draw [tension=0.80, dashed,gray!80] plot [smooth] coordinates { (T) (-8.4,0.9) (T2)};
	\draw [gray!80,fill=gray!10, tension=0.80]
	plot [smooth] coordinates { (-3.,-3) (-2,-2.6) (-0.5,-2.5) (0.6,-3)}  -- %
	plot [smooth ] coordinates {(0.6,-3) (2,-1.8) (2.1,0.6) (3.6,2.8) } --
	plot [smooth ] coordinates {(3.6,2.8) (2,3.2)  (0.5,3.2)  (-0.2,3.7) } --
	plot [smooth ] coordinates { (-0.2,3.7) (-2.3,1) (-3.,-3) };
	
	\draw [transparent, tension=0.80, blur shadow={shadow blur radius=6ex,shadow scale=0.6, shadow blur steps=12, shadow opacity=12}]
	plot [smooth] coordinates { (-3.,-3) (-2,-2.6) (-0.5,-2.5) (0.6,-3)}  -- 
	plot [smooth ] coordinates {(0.6,-3) (2,-1.8)   } --
	plot [smooth ] coordinates { (2,-1.8)  (2.1,0.6) (3.6,2.8) } --
	plot [smooth ] coordinates {(3.6,2.8) (2,3.2)  (0.5,3.2)  (-0.2,3.7) } --
	plot [smooth ] coordinates { (-0.2,3.7) (-2.3,1) (-3.,-3) };

	\draw [black,fill=gray!40] plot (-2.0,-0.5) -- (-0.5,2.5) -- (2.1,2.5) -- (0.6,-0.5) -- (-2.0,-0.5);
	\draw node[black] at (-2.25,-2.3) {$h(\mathcal{W}_0)$};
	\coordinate (X) at (0,1);
	\draw[rotate=-50,color=red!100] (X) ellipse (0.15cm and 0.3cm);
	\draw[rotate=-50,color=red!50] (X) ellipse (0.3cm and 0.6cm);
	\draw[rotate=-50,color=red!10] (X) ellipse (0.45cm and 0.9cm);
	\draw (X) node {$\times$};
	\draw[below] (X) node {$h(w_0)$};

	\draw[thick,->]	(-6.7,1) to[out=20,in=160] node[above] {$h$}	(-1.8,1);

	\draw node[black] at (-1.0,-0.2) {$\bar h(\mathcal{W}_0)$};
	
	\coordinate (Y) at (1.2,2.65);
	\draw     (Y) node {$\times$};
	\draw[above] (Y) node {$y^*$};
		
	\coordinate (pY) at (1.4,1.9);
	\coordinate (ppY) at (1.85,1.4);
	\draw         (pY) node {$\bullet$};
	\draw[gray!80]         (ppY) node {$\bullet$};
	\draw[dotted] 	(Y)	-- 	(pY);
	\draw[dotted] 	(Y)	-- 	(ppY);
	
	\draw [tension=0.80,dashed] plot [smooth] coordinates { (X) (0.6,1.6) (pY)};
	\draw ($ (pY) + (-.12,-.02) $) -- ($ (pY) + (-0.14,0.07) $) -- ($ (pY) + (-0.04,0.1) $);
	
	\draw [tension=0.80, dashed,gray!80] plot [smooth] coordinates { (X) (0.7,1.4) (ppY)};
	\draw [gray!80]($ (ppY) + (-.11,.01) $) -- ($ (ppY) + (-0.14,0.1) $) -- ($ (ppY) + (-0.04,0.1) $);
\end{tikzpicture}
\caption{There is a small neighborhood $\Ww_0 \subset \RR^p$ of the initialization $w_0$, which image by $h$ is a differentiable manifold in $\Ff$. In the lazy regime, the optimization paths (both in $\mathcal{W}$ and in $\Ff$) for the non-linear model $h$ (dashed gray paths) are close to those of the linearized model $\bar h$ (dashed black paths) until convergence or stopping time (Section~\ref{sec:dynamics}). This figure illustrates the under-parameterized case where $p< \dim(\Ff)$.}\label{fig:manifold} 
\end{figure}

\paragraph{Step 1.}
Our proof is along the same lines as that of Theorem~\ref{th:over}, but performed in $\Pi_r$ which can be thought of as a straighten up version of $h(\Ww_0)$. Consider the function $G_\alpha$ defined for $g \in \bar \Ff_0$ as $G_\alpha (g) = R(\alpha \psi^{-1}(g))/\alpha^2$. The gradient and Hessian of $G_\alpha$ satisfy, for $v_1,v_2\in \RR^p$,
\begin{align*}
\nabla G_\alpha(g) &= \frac1\alpha (D{\psi(g)^{-1}})^\intercal \nabla R(\alpha \psi^{-1} (g)),\\
D^2 G_\alpha(g)(v_1,v_2) &= v_1^\intercal (D{\psi(g)^{-1}})^\intercal \nabla^2R(\alpha \psi^{-1}(g))D{\psi(g)^{-1}}v_2 \\
&\quad + \frac1\alpha D^2{\psi(g)^{-1}}(v_1,v_2)^\intercal \nabla R(\alpha \psi^{-1}(g)) .
\end{align*}
The second order derivative of $G_\alpha$ is the sum of a first term with eigenvalues in an interval $[C^{-1},C]$, and a second term that goes to $0$ as $\alpha$ increases.  It follows that  $G_\alpha$ is smooth and strongly convex for $\alpha$ large enough.
Note that if $R$ or $\psi^{-1}$ are not twice continuously differentiable, then the Hessian computations should be understood in the distributional sense (this is sufficient because the functions involved are Lipschitz smooth).
Also, let $g^*$ be a minimizer of the lower-semicontinuous closure of $G_\alpha$ on the closure of $\bar \Ff_0$. By strong convexity of $R$ and our assumptions, it holds
\[
\Vert g^*\Vert^2 \leq \frac2{m} (G_\alpha(0)- G_\alpha(g^*))\leq \frac{2R(0)}{\alpha^2 m},
\]
so $g^*$ is in the interior of $\bar \Ff_0$ for $\alpha$ large enough and is then the unique minimizer of $G_\alpha$.

\paragraph{Step 2.} Now consider $T:= \inf\{ t\geq 0\;;\; w_\alpha(t) \notin \Ww_0\}$. For $t\in [0,T)$, the trajectory $w_\alpha(t)$ of the gradient flow~\eqref{eq:lazyGF} has ``mirror'' trajectories in the four spaces in the diagram above. Let us look more particularly at $g(t) \coloneqq \pi_r\circ \varphi(w_\alpha(t)) = \psi \circ h(w_\alpha(t))$ for $t<T$. In the following computation, we write $D\varphi$ for the value of the differential  at the corresponding point of the dynamic $D\varphi(w_\alpha(t))$ (and similarly for other differentials). By noticing that $Dh= D{\psi^{-1}}D{\pi_r}D{\varphi}$, we have
\begin{align*}
g'(t) &= - \frac1\alpha D{\psi} Dh Dh^\intercal \nabla R(\alpha \psi^{-1}(g(t)) \\
&= - \frac1\alpha D{\pi_r}D{\varphi}D{\varphi}^\intercal D{\pi_r}^\intercal (D{\psi^{-1}})^\intercal \nabla R(\alpha \psi^{-1}(g(t)).
\end{align*}
so $g(t)$ remains in $\Pi_r$. Also, the first $r\times r$ block of $D{\pi_r}D{\varphi}D{\varphi}^\intercal D{\pi_r}^\intercal$ is positive definite on $\Pi_r$, with a positive lower bound (up to taking $\Ww_0$ and $\Ff_0$ smaller if necessary). Thus by Lemma~\ref{lem:stronglyconvexGF}, there are constants $C_1,C_2>0$ independent of $\alpha$ such that, for $t \in [0,T)$,
$
\Vert g(t)-g^*\Vert \leq C_1\Vert g(0) - g^*\Vert \exp\left( -C_2 t\right).
$
\paragraph{Step 3.} Now we want to show that $T=+\infty$ for $\alpha$ large enough. It holds
\begin{align*}
w'(t) = -\frac1\alpha Dh^\intercal\nabla R(\alpha h(w_\alpha(t))=  D\varphi^\intercal D{\pi_r}^\intercal \nabla G_\alpha(g(t))
\end{align*}
and, by Lipschitz-smoothness of $G_\alpha$ (Step 1), $\Vert \nabla G_\alpha(g(t))\Vert \leq \frac{C}{\alpha}\Vert g(t)-g^*\Vert$ hence
\[
\Vert w_\alpha(t)-w_0\Vert \leq \frac C\alpha \int_0^t \exp(-C_2s)ds\leq \frac{C}{\alpha C_2}.
\]
Thus, by choosing $\alpha$ large enough, one has  $w_\alpha(t)\in \Ww_0$ for all $t\geq 0$, so $T=\infty$ and the theorem follows. 

\section{Experimental details and additional results}\label{app:experiments}

\subsection{Many neurons dynamics visualized} 
The setting of Figure~\ref{fig:parameters} is the same as for panels (a)-(b) in Figure~\ref{fig:cover} except that $m=200, n=200$: it allows to visualize behavior of the training dynamics for a larger number of neurons. Symmetrized initialization to set $f(w_0,\cdot)=0$ was used on panel~(c) but not on panel~(b), where we see that the neurons need to move slightly more in order to compensate for the non-zero initialization. As on Figure~\ref{fig:cover}, we observe a good behavior in the non-lazy regime for small $\tau$.

\begin{figure}[h]
\centering
\begin{subfigure}{0.33\linewidth}
\centering
\includegraphics[scale=0.35]{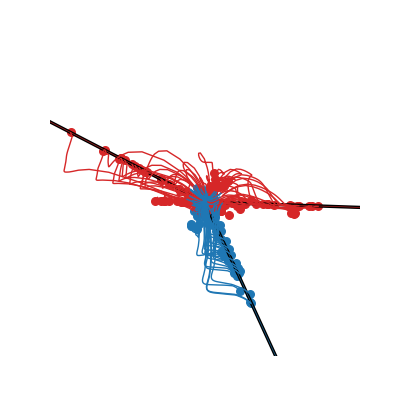}


\caption{Non-lazy training ($\tau=0.1$)}
\end{subfigure}%
\begin{subfigure}{0.33\linewidth}
\centering
\includegraphics[scale=0.35]{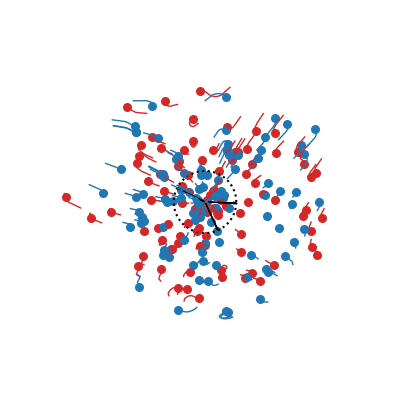}


\caption{Lazy ($\tau=2$, not symmetrized)}
\end{subfigure}%
\begin{subfigure}{0.33\linewidth}
\centering
\includegraphics[scale=0.35]{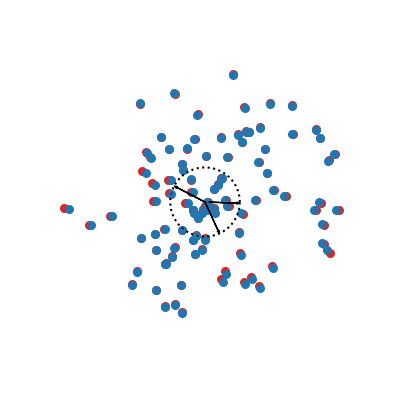}


\caption{Lazy ($\tau=2$, symmetrized)}
\end{subfigure}
\caption{Training a two-layer ReLU neural network initialized with normal random weights of variance $\tau^2$, as in Figure~\ref{fig:cover}, but with more neurons. In this $2$-homogeneous setting, changing $\tau^2$ is equivalent to changing $\alpha$ by the same amount so lazy training occurs for large $\tau$. }
\label{fig:parameters}
\end{figure}

\subsection{Stability of activations} 
We define here the ``stability of activations'' mentioned in Section~\ref{sec:CNNexpe}. We consider a ReLU layer $\ell$ of size $n_\ell$ in a neural network and the test input data $(x_i)_{i=1}^N$ (the test images of CIFAR10 in our case). We call $z_{ij}(T) \in \RR$ the value of the pre-activation (i.e.\ the value that goes through the ReLU function as an input) of index $j$ on the data sample $i$, obtained with the parameters of the network at epoch $T$. The ``stability of activations'' for this layer is defined as $s_\ell \coloneqq \frac{Q}{n_\ell\times N}$ where $L$ is the number of ReLU layers, $Q$ is the number of indices $(i,j)$ that satisfy $\sign(z_{ij}(T_{last})) = \sign(z_{ij}(T_{init}))$ for $i\in \{1,\dots, B\}$ and $j\in \{1,\dots,n_\ell\}$, where $T_{init}$ refers to initialization and $T_{last}$ to the end of training. The quantity that we report on Figure~\ref{fig:CNNs}(a) is the average of $s_\ell$ over all ReLU layers of the VGG-11 network, for various values of $\alpha$.

\subsection{Spectrum of the tangent kernel}
In the setting of Figure~\ref{fig:CNNs}(a), we want to understand why the linearized model (that is, trained for large $\alpha$) could not reach low training accuracies in spite of being highly over-parameterized. Figure~\ref{fig:VGGspectrum}(a) shows the train and test losses after $70$ epochs where we see that the training loss is far from $0$ for all $\alpha\geq 10$. 
We report on Figure~\ref{fig:VGGspectrum}(b) the normalized and sorted eigenvalues  $\sigma_i^2$ of the tangent kernel $Dh(w_0)Dh(w_0)^\intercal$ (notice the log-log scale) evaluated for two distinct input data sets $(x_i)_{i=1}^n$ of size $n=500$: (i) images randomly sampled from the training set of CIFAR10 and (ii) images with uniform random pixel values.  Since there are $10$ output channels, the corresponding space $\Ff$ has $10\times 500$ dimensions. We observe that there is a gap of $1$ order of magnitude between the $0.2\%$ largest eigenvalues and the remaining ones---which causes the ill conditionning---and then a decrease of order approximately $O(1/i)$. We observe a similar pattern with the CIFAR10 inputs and completely random inputs, which suggests that this conditioning is intrinsic to the linearized VGG model. Note that modifying the neural network architecture to improve this conditioning, or using optimization methods that are better adapted to ill-conditionned models, is beyond the scope of the present paper.

\begin{figure}[h]
\centering
\begin{subfigure}{0.4\linewidth}
\includegraphics[scale=0.5]{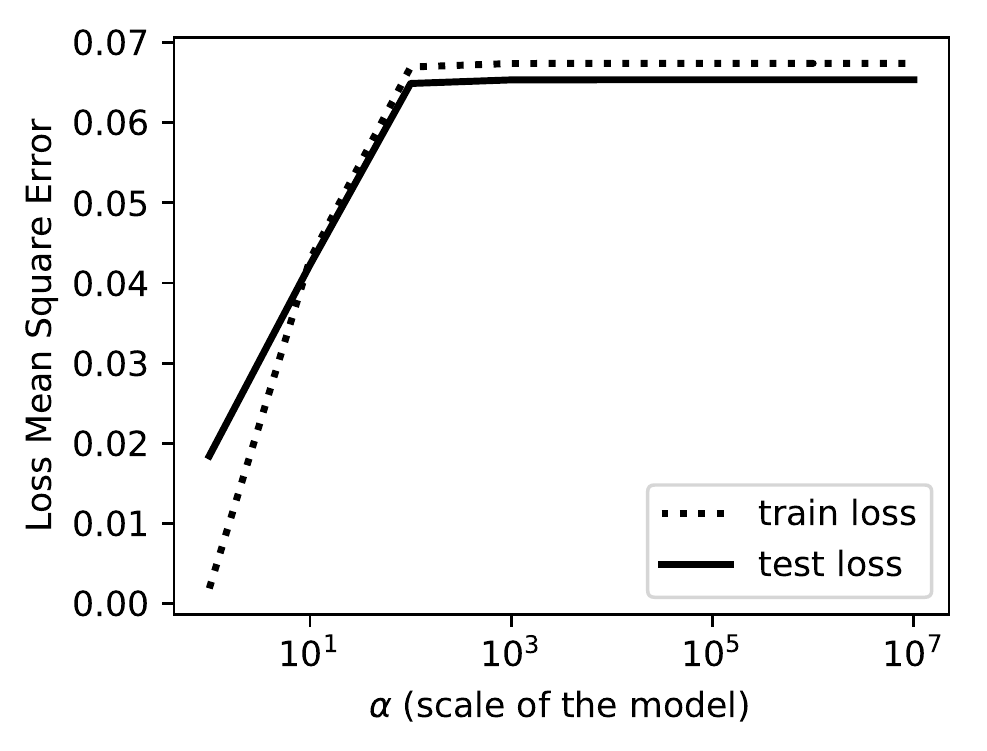}
\caption{}
\end{subfigure}%
\begin{subfigure}{0.4\linewidth}
\includegraphics[scale=0.5]{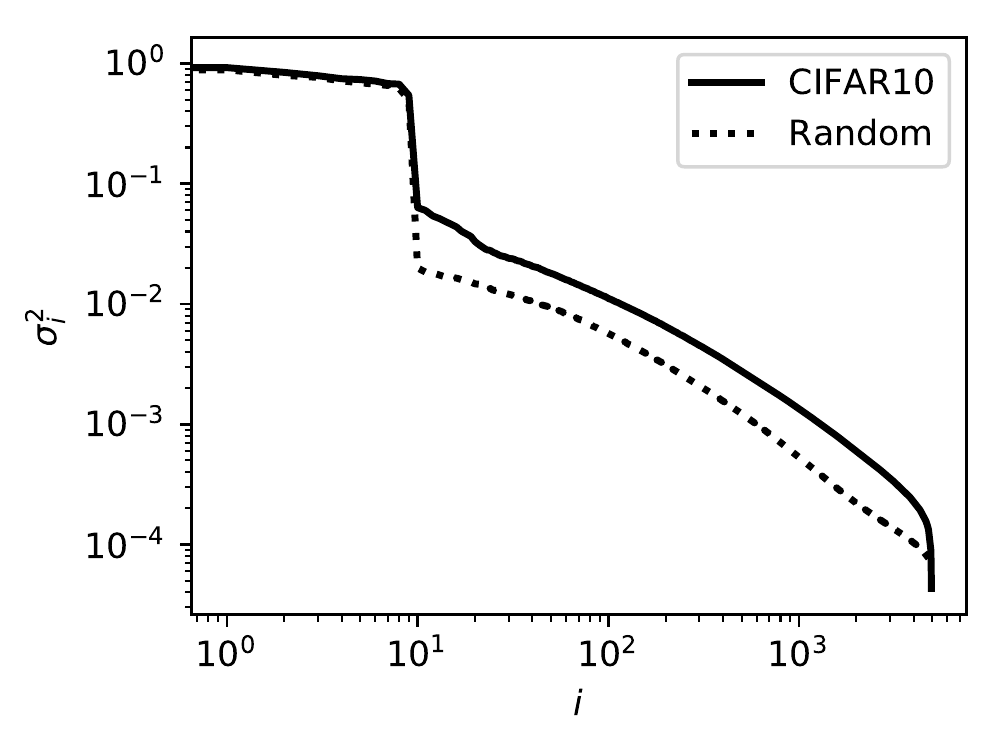}
\caption{}
\end{subfigure}
\caption{(a) End-of training train and test loss. (b) Spectrum of the tangent kernel $Dh(w_0)Dh(w_0)^\intercal$ for the VGG11 model on two data sets. }
\label{fig:VGGspectrum}
\end{figure}

\end{document}